\documentclass[11pt]{article}
\usepackage{amscd}
\usepackage{amssymb}
\usepackage{amsthm}
\usepackage{amsfonts}
\usepackage{amsmath}
\usepackage{mathtools}
\usepackage{mathrsfs}
\usepackage{latexsym}
\usepackage{graphicx}
\usepackage{enumerate}
\usepackage{geometry}
\usepackage{layout}
\usepackage[bottom]{footmisc}
\newtheorem{theorem}{Theorem}
\theoremstyle{plain}
\newtheorem{corollary}[theorem]{Corollary}

\newtheorem{lemma}[theorem]{Lemma}

\newtheorem{proposition}[theorem]{Proposition}

\numberwithin{equation}{section}
\numberwithin{theorem}{section}
\newcommand{\<}{\langle}
\renewcommand{\>}{\rangle}
\newcommand{\x}{\otimes}
\newcommand{\R}{\mathbb{R}}
\newcommand{\E}{\mathbb{E}}
\newcommand{\bF}{\mathbb{F}}
\newcommand{\N}{\mathbb{N}}

\newcommand{\bW}{\mathbb{W}}
\newcommand{\bD}{\mathbb{D}}
\newcommand{\bP}{\mathbb{P}}
\newcommand{\mH}{\mathcal{H}}

\newcommand{\mT}{\mathcal{T}}
\newcommand{\mF}{\mathcal{F}}

\newcommand{\mS}{\mathcal{S}}
\newcommand{\mC}{\mathcal{C}}
\newcommand{\mA}{\mathcal{A}}
\newcommand{\mU}{\mathcal{U}}
\newcommand{\mI}{\mathcal{I}}
\newcommand{\mJ}{\mathcal{J}}
\newcommand{\mV}{\mathcal{V}}

\newcommand{\sD}{\mathscr{D}}

\def\e{{\mathrm{e}}}
\allowdisplaybreaks

\begin{document}
\title{Sensitivity analysis in the infinite dimensional Heston model}
\date{\today}
\author{Fred Espen Benth\thanks{Department of Mathematics, University of Oslo, Oslo, Norway.} \\ \texttt{fredb@math.uio.no} \and Giulia Di Nunno\footnotemark[1]\,\,\thanks{Department of Business and Management Science, NHH Norwegian School of Economics, Bergen, Norway.} \\ \texttt{giulian@math.uio.no} \and Iben Cathrine Simonsen\footnotemark[1] \\ \texttt{ibens@math.uio.no}}



\maketitle

\begin{abstract}
We consider the infinite dimensional Heston stochastic volatility model proposed in \cite{BS}.
The price of a forward contract on a non-storable commodity is modelled by a generalized Ornstein-Uhlenbeck process in the Filipovi\'{c} space with this volatility.
We prove different representation formulas for the forward price.
Then we consider prices of options written on these forward contracts and we study
sensitivity analysis with computation of the Greeks with respect to different parameters in the model.
Since these parameters are infinite dimensional, we need to reinterpret the meaning of the Greeks.
For this we use infinite dimensional Malliavin calculus and a randomization technique.
\end{abstract} 

{\footnotesize\noindent{\em MSC 2020:} 60H07, 60H15, 91G20, 46N30 }

\section{Introduction}
Consider a forward contract on electricity or some other non-storable commodity.
The forward price is often represented as a random field $f:[0,T]\times \R^+\times\Omega\to\R$, where $T<\infty$ and $(\Omega,\mF,\bP)$ is a complete probability space equipped with a $\bP$-augmented filtration $\bF=\{\mF_t\}_{0\leq t \leq T}$. In the Musiela notation, $t\in[0,T]$ is the time horizon and $x=T-t\in\R^+$ represents the time to maturity. 
Clearly, we can regard the forward price as an infinite dimensional stochastic process $f:[0,T]\times\Omega\to H$, where $H$ is a suitable space of real-valued functions on $\R^+$. This is the approach taken to describe the dynamics of evolution of the forward prices.

In this paper we consider an infinite dimensional stochastic volatility model for the forward price. 
Our interests are motivated by the power markets where there are clear signs of intertemporal correlation structures across time to maturity (see e.g. \cite{BSBK}) and also signs of non-Gaussianity (see e.g. \cite{BP}).
We choose to consider the Heston-type infinite dimensional volatility model proposed in \cite{BS}, and study both the pricing of options written on the forwards and provide tools for the sensitivity analysis. 
For this last one, we have to introduce the adequate concepts in the infinite dimensional setting.
We resolve this by exploiting some ideas in the approach of \cite{BDHP}, using the interplay of functional derivatives and the Malliavin calculus via a form of randomization.
We remark that there is a structural difference in the application perspectives between our paper and \cite{BDHP}.
While we face an infinite dimensional problem all the way through, the authors in \cite{BDHP} consider a finite dimensional noise and path-dependent coefficients. 

Our paper is organized as follows. In Section \ref{section:vol} we introduce the forward price model.
In Section \ref{section:sensitivity} we consider prices of options written on such forward contracts and study the sensitivity of these prices with respect to different model parameters.

\section{Stochastic volatility forward price model} \label{section:vol}
Let $H$ be a separable Hilbert space with inner product $\<\cdot,\cdot\>_H$ and associated norm $\|\cdot\|_H$.
The space of bounded linear operators from $H$ into itself is denoted by $L(H)$.
It is a Banach space with the operator norm denoted by $\|\cdot\|_{\text{op}}$.
Furthermore, let $\mH\subseteq L(H)$ denote the space of Hilbert-Schmidt operators on $H$. 
Recall that $\mH$ is a separable Hilbert space.
The inner product and associated norm in $\mH$ are denoted by $\<\cdot,\cdot\>_{\mH}$ and $\|\cdot\|_{\mH}$, respectively. 

We start by introducing the stochastic volatility model from \cite{BS}.
Let $W=\{W_t\}_{t\geq 0}$ and $B=\{B_t\}_{t\geq 0}$ be independent $\bF$-adapted Wiener processes in $H$ with covariance operators $Q_W$ and $Q_B$.
Let $\mA$ and $\mC$ be densely defined operators on $H$ which generate $C_0$-semigroups $\{\mU_t\}_{t\geq 0}$ and $\{\mS_t\}_{t\geq 0}$ respectively, and let $\eta\in L(H)$.
Let $Z=\{Z_t\}_{t\geq 0}$ be an $\bF$-adapted stochastic process in $H$ satisfying $\|Z_t\|_H=1$ for all $t\geq 0$.
The process $X=\{X_t\}_{t\geq 0}$, which will later be used to model the forward price, is defined by the following equation:
\begin{equation}\label{YX}
\left\{
\begin{alignedat}{2}
 	&dX_t = \mathcal C X_tdt + \Gamma^Z_t dB_t, && \mspace{40mu} X_0=x_0 \in H,
 	\\
 	&dY_t = \mathcal{A} Y_t\, dt + \eta\, dW_t, && \mspace{40mu} Y_0=y_0 \in H,
\end{alignedat}
\right.
\end{equation}
where $\Gamma^Z_t := Z_t\x Y_t$ defines the volatility process $\Gamma_Z=\{\Gamma^Z_t\}_{t\geq 0}$ which is $\bF$-adapted and takes values in $\mH$.
From Peszat and Zabczyk~\cite[Sect.~9.4]{PZ} we know that \eqref{YX} has a unique mild solution given by
\begin{equation}\label{mild-sol}
\left\{
\begin{alignedat}{2}
 	&X_t = \mS_t x_0 + \int_0^t \mS_{t-s}\Gamma^Z_s\,dB_s, \qquad &t\geq 0,
 	\\
 	&Y_t = \mathcal{U}_t y_0 + \int_0^t \mathcal{U}_{t-s}\eta\, dW_s, \qquad &t\geq 0,
\end{alignedat}
\right.
\end{equation}
since the stochastic integral $\int_0^t \Gamma^Z_s dB_s$ is well-defined by \cite{BS}.
The stochastic integral $\int_0^t \mS_{t-s}\Gamma^Z_s\,dB_s$ is also well-defined, since $\mS_t \in L(H)$ and its operator norm grows at most exponentially by the Hille-Yoshida Theorem, see Engel and Nagel~\cite[Prop.~I.5.5]{EN}.

The variance process $\mV=\{\mV_t\}_{t\geq 0}$ is defined by $\mV_t\coloneqq Y_t^{\x 2} = Y_t\x Y_t = \<Y_t,\,\cdot\,\>_H Y_t$, $t\geq 0$.
It is a process of symmetric and positive definite operators in $\mH$, and hence it has a unique square root process, which is given by
\begin{equation*}
\mV^{1/2}_t=\left\{\begin{array}{cl}\|Y_t\|_H^{-1}\mV_t, & Y_t\neq 0 \\ 0, & Y_t=0.\end{array}\right. 
\end{equation*}
The variance process can be decomposed as $\mV_t = \Gamma^Z_t(\Gamma^Z_t)^*$ for all $t\geq 0$, where $(\Gamma^Z_t)^* = Y_t\x Z_t$.
Clearly, several choices of  $\{Z_t\}_{t\geq 0}$ are possible.
A simple choice of $\{Z_t\}_{t\geq 0}$ is $Z_t=\gamma\in H$ for all $t\geq 0$, where $\|\gamma\|_H=1$.  
We could also define $\{Z_t\}_{t\geq 0}$ as $Z_t=Y_t/\|Y_t\|_H$ when $Y_t\neq 0$ and $Z_t=0$ when $Y_t=0$ to get $\Gamma^Z_t = \mV^{1/2}_t$ for all $t\geq 0$.


We have a full description of the characteristic functional of $X_t$ for $t\geq 0$ (result presented in \cite[Proposition 9]{BS}):
\begin{proposition} \label{charfunc_X}
Assume that $Z$ is $\bF^Y$-adapted, where $\bF^Y$ is the filtration generated by $Y$.   
Then, for any $h\in H$ and $t\geq 0$,
\begin{align*}
\varphi_{X_t}(h) &\coloneqq\E\left[\e^{\mathrm{i}\left<X_t,h\right>_H}\right] \\
&= \e^{\mathrm{i}\left<\mS_t x_0,h\right>_H}\E\left[\exp\left(-\frac{1}{2}\left\<\left(\int_0^t\|Q_B^{1/2}Z_s\|_H^2\mS_{t-s}\mV_s\mS^*_{t-s}\,ds\right)h,h\right\>_H\right)\right],
\end{align*}
where the integral above is a Bochner integral in $L(H)$.
\end{proposition}
From the proposition above, we see that for any $t\geq 0$, $X_t$, conditional on $\mF_t^Y$, is a Gaussian $H$-valued random variable. 
The expectation of $X_t$ is $\E[X_t]=\mS_t x_0$ and its covariance operator $Q_{X_t}$ is found in \cite{BS}:
\begin{equation} \label{eq:covar-X_t}
Q_{X_t}=\int_0^t\mS_{t-s}\E\left[\|Q^{1/2}_BZ_s\|_H^2\mV_s\right]\mS^*_{t-s}\,ds.
\end{equation} 
Hence $X$ is an $H$-valued conditionally Gaussian process.
In the particular case when $Z\equiv\gamma\in H$ with $\|\gamma\|_H=1$, we have that 
\begin{equation*}
Q_{X_t}=\int_0^t\mS_{t-s}\left((\mU_s y_0)^{\x2}+\int_0^s\mU_u\eta Q_W\eta^*\mU^*_u\,du\right)\mS^*_{t-s}\,ds,
\end{equation*}
which is obtained by applying Theorem 8.7(iv) of Peszat and Zabczyk~\cite{PZ}, see \cite{BS}.

\subsection{The forward price model}
We shall model the price $F(t,T)$ at time $t$ of a forward contract maturing at time $T$ by choosing a specific Hilbert space $H_w$ and evaluating $X_t$ at the time to maturity.
The price $F(t,T)$ will be called the forward price.
We will use the Musiela notation in which the forward price is written as $f(t,x)=F(t,t+x)$ where $x=T-t$ is the time to maturity.

The Filipovi\'{c} space $H_w$ was first introduced in \cite{filipovic} and is a Hilbert space consisting of real-valued measurable functions.
Let $w:\R_+\rightarrow[1,\infty)$ be a measurable and increasing weight function with $w(0)=1$ and $\int_0^{\infty} w^{-1}(x)\,dx<\infty$.
Then $H_w$ is defined as the space of functions $h\in L^1_{\text{loc}}(\R^+)$ which possess a weak derivative $h'\in L^1_{\text{loc}}(\R^+)$ such that $\int_0^{\infty} |h'(x)|^2 w(x)\,dx<\infty$. 
We know that every such function has an absolutely continuous version.
The norm on $H_w$ is defined as 
\begin{equation*}
\|h\|_w^2 \coloneqq |h(0)|^2 + \int_0^{\infty} |h'(x)|^2 w(x)\,dx.
\end{equation*}

The evaluation functional $\delta_x:H_w\rightarrow\R_+$ defined by $\delta_x(h)\coloneqq h(x)$ is a continuous linear functional on $H_w$ for all $x\in\R_+$, see \cite{filipovic}.
This means that $\delta_x\in H_w^*$, and hence by the Riesz representation theorem $\delta_x = \<\cdot,h_x\>_w$ for some element $h_x\in H_w$.
Let $\|\cdot\|_*$ denote the norm on $H_w^*$.
From \cite[Lemma 3.11]{filipovic}, we have
\begin{equation} \label{h_x}
 h_x(y) = 1 + \int_0^{x\wedge y} \frac{1}{w(s)}ds, \quad y\in\R_+,
\end{equation}
and we see that for $h\in H_w$,
\begin{align*}
\<h,h_{x}\>_w &= h(0)h_{x}(0) + \int_0^{\infty} w(y)h'(y)h_{x}'(y)\,dy = h(0) + \int_0^{\infty} w(y)h'(y)\frac{1}{w(y)}1_{\{y\leq x\}}\,dy \\
&= h(0) + \int_0^{x} h'(y)\,dy = h(x) = \delta_{x}(h).
\end{align*}
The following lemma is part of a result in Benth and Kr\"uhner~\cite[Lemma 3.1]{BK-HJM}.

\begin{lemma}
Let $\delta_x$ be the evaluation functional on $H_w$.
Then $\|\delta_x\|_*^2 = h_x(x)$, where $h_x$ is given by equation \eqref{h_x}. 
\end{lemma}

Consider now equation \eqref{YX} with $H$ being the Filipovi\'{c} space $H_w$ and $\mC$ being the derivative operator $\partial/\partial x$.
The $C_0$-semigroup generated by the derivative operator on $H_w$ is the semigroup of left-shift operators, i.e. $\mS_x(h)=h(\cdot+x)$ for $x\geq 0$.
It is shown in Filipovi\'{c}~\cite[Equation (5.10)]{filipovic} that for $x\geq 0$, $\|\mS_x\|_{\text{op}}\leq C$ for some constant $C$.
A value for the constant $C$ is found in Benth and Kr\"uhner~\cite[Lemma 3.4]{BK-2015}:

\begin{lemma}\label{shift-op-bound}
For $x\geq 0$, it holds that $\|\mS_x\|_{\text{op}}\leq \sqrt{2 \max \left(1,\int_0^\infty \frac{1}{w(s)}\,ds\right)}$.
\end{lemma}

While on $L^2(\R)$ the adjoint of the left-shift operator is the right-shift operator, this is not the case on $H_w$.
The following lemma gives the adjoint of the left-shift operator on $H_w$.

\begin{lemma}
For $x\geq 0$, let $\mS_x$ denote the left-shift operator on $H_w$, defined by $\mS_x h=h(\cdot+x)$. 
The adjoint operator $\mS_x^*$ of $\mS_x$, defined by the relation $\<\mS_x f,g\>_w = \<f,\mS_x^* g\>_w$ for all $f,g\in H_w$, is given by
\begin{equation}
\label{shift-adjoint}
\mS_x^* g(y) = 
\left\{
\begin{alignedat}{2}
 	&g(0)\left(1 + \int_0^y \frac{1}{w(s)}\,ds\right), \qquad &0\leq y \leq x,
 	\\
 	&g(0)\left(1 + \int_0^x \frac{1}{w(s)}\,ds\right) + \int_0^{y-x} \frac{w(s)}{w(s+x)}g'(s)\,ds, \qquad &y>x.
\end{alignedat}
\right.
\end{equation}
\end{lemma}

\begin{proof}
From Proposition 3.8 in \cite{BK-2015}, we have that for an operator $\mT\in L(H_w)$, the adjoint operator $\mT^*$ is given by
\begin{equation*}
\mT^* g(y) = g(0)\eta(y) + \int_0^\infty q(y,s)g'(s)\,ds,
\end{equation*}
where
\begin{equation*}
\eta(y)\coloneqq (\mT h_y)(0)
\end{equation*}
and 
\begin{equation*}
q(y,s)\coloneqq (\mT h_y)'(s)w(s).
\end{equation*}
For the left-shift operator we find that
\begin{equation*}
\eta(y) = (\mS_x h_y)(0) = h_y(x) = \int_0^{y\wedge x}\frac{1}{w(s)}\,ds,\qquad y\geq 0.
\end{equation*}
From \cite[p. 78]{filipovic} we have that $(\mS_x h)'=\mS_x h'$. 
Hence, for $y,s\geq 0$,
\begin{align*}
q(y,s) &= (\mS_x h_y)'(s) w(s) = \mS_x h_y'(s) w(s) \\
&= \mS_x\left(\frac{1}{w(s)}1_{\{s\leq y\}}\right) w(s) = \frac{1}{w(s+x)}1_{\{s+x\leq y\}} w(s)\\
&= \frac{w(s)}{w(s+x)}1_{\{s\leq y-x\}}.
\end{align*}
This gives
\begin{equation*}
\mS_x^* g(y) = g(0)\left(1+\int_0^{y\wedge x}\frac{1}{w(s)}\,ds\right) + \int_0^{(y-x)\vee 0} \frac{w(s)}{w(s+x)}g'(s)\,ds,
\end{equation*}
and the result follows.
\end{proof}

We define the forward price as 
\begin{equation}\label{musiela}
f(t,x)\coloneqq \delta_x(X_t), \qquad t\in[0,T],\,x\in\R^+.
\end{equation}
Observe that the covariance between the prices of two forward contracts with different times to delivery follows from Proposition \ref{charfunc_X}, see \cite{BS}.

\begin{corollary}
For all $x,y\in\mathbb{R}_+$, we have
\begin{align*}
\textnormal{Cov}\big(f(t,x),\, f(t,y)\big)=\E\left[\int_0^t\|Q_B^{1/2}Z_s\|_w^2\,\delta_{x+t-s}(Y_s)\delta_{y+t-s}(Y_s)\,ds\right].
\end{align*}
\end{corollary}
In the special case when $Z_t=\gamma\in H_w$ with $\|\gamma\|_w=1$, it is shown in \cite{BS} 
that
\begin{align*}
\textnormal{Cov}(f(t,x),f(t,y))&=\|Q_B^{1/2}\gamma\|_w^2\int_0^t\delta_{y+t-s}(\mU_sy_0)^{\otimes 2}\delta_{x+t-s}^*(1)\,ds \\
&\qquad+\|Q_B^{1/2}\gamma\|_w^2\int_0^t\delta_{y+t-s}\left(\int_0^s\mU_u\eta Q_W\eta^*\mU^*(u)\,du\right)\delta^*_{x+t-s}(1)\,ds,
\end{align*}
by applying Theorem 8.7(iv) of Peszat and Zabczyk~\cite{PZ}.

\subsection{Commodities with delivery period}
In the case of commodities with delivery over a period of time $[T_1, T_2]$, as for electricity, the forward price $G(t,T_1,T_2)$ at time $t\leq T_1$ is modelled by
\begin{equation*}
 G(t,T_1,T_2) \coloneqq \frac{1}{T_2-T_1} \int_{T_1}^{T_2} F(t,T)\,dT.
\end{equation*}

Hereafter, we study different representations of the forward price.
For this we introduce two integration functionals, $\mI_d$ and $\mJ_{x,d}$.
For $d\geq 0$, $\mI_d:H_w\rightarrow \R$ is defined as
\begin{equation}\label{I}
\mI_d(h)\coloneqq \frac{1}{d}\int_0^d h(u)\,du, \qquad h\in H_w.
\end{equation}
For $x,d\geq 0$, $\mJ_{x,d}:H_w\rightarrow \R$ is defined as 
\begin{equation}\label{J}
\mJ_{x,d}(h)\coloneqq \frac{1}{d}\int_0^d \delta_{x+u}(h)\,du, \qquad h\in H_w.
\end{equation}
We can see that $\mJ_{x,d} = \mI_d\circ \mS_x$, where $\mS_x$ is the left-shift operator on $H_w$. 
Using the Musiela notation we let $x$ denote the time to the start of the delivery period and $d$ denote the length of the delivery period.
Then $[T_1,T_2]=[t+x,t+x+d]$, and the forward price can be written as
\begin{equation}\label{g}
g(t,x,d)\coloneqq \mJ_{x,d}(X_t) = \frac{1}{d} \int_0^d f(t,x+u)\,du, \qquad t\in[0,T_1],\, x\in\R^+, d\in\R^+.
\end{equation}
The covariance between the prices of two forward contracts with different times $x$ and $y$ to the start of the delivery period and different lengths $d_1$ and $d_2$ of the delivery period was found in \cite{BS}:
\begin{equation*}
\textnormal{Cov}\big(g(t,x,d_1),\, g(t,y,d_2)\big)=\frac{1}{d_1 d_2}\int_0^{d_1}\int_0^{d_2}\int_0^t\E\left[\|Q_B^{1/2}Z_s\|_w^2\,\delta_{x+t-s+u}(Y_s)\delta_{y+t-s+v}(Y_s)\right]\,ds\,dv\,du.
\end{equation*}

The next two results present some properties of the functionals $\mI_d$ and $\mJ_{x,d}$ defined in \eqref{I} and \eqref{J}.

\begin{lemma}
The integration functionals $\mI_d$ and $\mJ_{x,d}$ defined in \eqref{I} and \eqref{J} are well-defined continuous linear functionals on $H_w$, i.e. $\mI_d, \mJ_{x,d}\in H_w^*$.
\end{lemma}

\begin{proof}
Since $h\in H_w$ is locally integrable, $\mI_d$ is a well-defined linear functional on $H_w$.
We show that it is bounded.
For $h\in H_w$,
 \begin{equation*}
|\mI_d(h)| \leq \frac{1}{d}\int_0^d|\delta_{u}(h)|\,du \leq \frac{1}{d}\left(\int_0^d \|\delta_{u}\|_*\,du\right)\|h\|_w
\end{equation*}
holds.
Since $x\mapsto \|\delta_x\|_*$ is locally integrable, it follows that $\mI_d$ is bounded, i.e. $\mI_d\in H_w^*$.
Since $\mS_t$ is continuous, it follows that $\mJ_{x,d}\in H_w^*$.
\end{proof}

\begin{lemma}
\label{h_x,d}
For any $h\in H_w$, we have that
\begin{equation*}
\mJ_{x,d}(h) = \<h, h_{x,d}\>_w,
\end{equation*}
where 
\begin{equation}
\label{eq:h_x,d}
h_{x,d}(y) = 
\left\{
\begin{alignedat}{3}
 	&1+\int_0^y\frac{1}{w(s)}\,ds, \qquad &0\leq y \leq x,
 	\\
 	&1+\int_0^x\frac{1}{w(s)}\,ds + \int_0^{y-x} \frac{d-s}{dw(s+x)} \,ds, \qquad &x<y\leq x+d,
 	\\
 	&1+\int_0^x\frac{1}{w(s)}\,ds + \int_0^d \frac{d-s}{dw(s+x)} \,ds, \qquad &y>x+d.
\end{alignedat}
\right.
\end{equation}
\end{lemma}

\begin{proof}
Since $\mI_d\in H_w^*$, we have that $\mI_d = \<\cdot,h_d^{\mI}\>_w$ for some element $h_d^{\mI}\in H_w$.
From Filipovi\'{c} \cite[Lemma 3.11]{filipovic} we have
\begin{equation} 
\label{h_d^I}
h_d^{\mI}(y) = 1 + \frac{1}{d}\int_0^y \frac{d-s\wedge d}{w(s)}\,ds.
\end{equation}
Since also $\mJ_{x,d}\in H_w^*$, we have that $\mJ_{x,d} = \<\cdot,h_{x,d}\>_w$ for some $h_{x,d}\in H_w$.
We see that 
\begin{equation*}
\mJ_{x,d}(h) = \mI_d(\mS_x(h)) = \left\<\mS_x(h),h_d^{\mI}\right\>_w = \left\<h,\mS_x^*(h_d^{\mI})\right\>_w,
\end{equation*}
which gives $h_{x,d} = \mS_x^*(h_d^{\mI})$.
Since $h_d^{\mI}(0)=1$ and
\begin{equation*}
(h_d^{\mI})'(y) = \frac{1}{d}\frac{d-y\wedge d}{w(y)} = \frac{d-y}{dw(y)} 1_{\{0\leq y\leq d\}},
\end{equation*}
we get by \eqref{shift-adjoint} that
\begin{align*}
h_{x,d}(y) &= \mS_x^* h_d^{\mI}(y) \\
&= h_d^{\mI}(0)\left(1+\int_0^{y\wedge x}\frac{1}{w(s)}\,ds\right) + \int_0^{(y-x)\vee 0} \frac{w(s)}{w(s+x)}(h_d^{\mI})'(s)\,ds\\
&= 1+\int_0^{y\wedge x}\frac{1}{w(s)}\,ds + \int_0^{(y-x)\vee 0} \frac{w(s)}{w(s+x)}\frac{d-s}{dw(s)} 1_{\{0\leq s\leq d\}}\,ds\\
&= 1+\int_0^{y\wedge x}\frac{1}{w(s)}\,ds + \int_0^{((y-x)\vee 0)\wedge d} \frac{d-s}{dw(s+x)} \,ds.
\end{align*}
\end{proof}

The following corollary gives a representation for the forward price as an inner product between $X_t$ and an element of $H_w$.

\begin{corollary}
The forward price can be expressed as 
\begin{equation*}
g(t,x,d) = \<X_t, h_{x,d}\>_w,
\end{equation*}
where the function $h_{x,d}$ is given in \eqref{eq:h_x,d}.
\end{corollary}

\section{Sensitivity analysis} \label{section:sensitivity}

Consider an option written on a forward contract with delivery period $[T_1, T_2]$ as described in the previous section. 
Let $\tau$ denote the exercise time of the option, where $0\leq \tau\leq T_1$.
Let $\Phi:\R\rightarrow\R^{+}$ denote the payoff function, which is assumed to be measurable and of at most linear growth.
The price of the option at time $t\leq \tau$ is represented as
\begin{equation*}
\Pi_t = e^{-r(\tau-t)} \E\left[\Phi(G(\tau,T_1,T_2))\,|\,\mF_{t}\right],
\end{equation*}
where $r$ is a constant instantaneous interest rate.
Note that we are taking the expectation under the market probability measure $\bP$ and not a risk-neutral probability measure.
The reason for this is that electricity is not storable, hence it is not a tradeable asset in the usual sense, see \cite{BSBK}.
The standard argument of using a risk-neutral probability measure to prevent arbitrage opportunities is therefore not valid in this case, and any equivalent martingale measure can be used for option pricing.

As before we adopt the Musiela notation with $x=T_1-\tau$ and $d=T_2-T_1$. Recalling \eqref{g}, the option price can then be written as
\begin{equation}
\label{option-price}
 \Pi_t = e^{-r(\tau-t)} \E\left[\Phi(g(\tau,x,d))\,|\,\mF_{t}\right] = e^{-r(\tau-t)} \E\left[(\Phi\circ\mJ_{x,d})(X_\tau)\,|\,\mF_{t}\right],
\end{equation}
where the functional $\mJ_{x,d}$ is defined in \eqref{J}.
The option price at time $t=0$ is given by
\begin{equation}
 \Pi_0 = \e^{-r\tau}\E\left[(\Phi\circ\mJ_{x,d})(X_\tau)\right].
\end{equation}

In sensitivity analysis we are interested in the derivatives of the option price with respect to different parameters of the underlying price model.
In finance, these derivatives are called \emph{Greeks} since they are denoted by greek letters.
Examples are the delta, which is the derivative with respect to the initial value of the underlying asset, and the vega, which is the derivative with respect to the volatility.
Our framework presents "parameters" in infinite dimensions, e.g. the initial forward price is a function in $H_w$.
Thus we have to reinterpret the meaning of the Greeks.
For the delta, a natural choice is to take inspiration from \cite{BDHP} and interpret it as a directional derivative. 
For the vega however, there is no natural generalization to our framework.
We choose to compute the directional derivatives with respect to the "parameters" $y_0$ and $\eta$ of the stochastic volatility model instead.

We now consider the option price at time $t=0$ as a function of the initial value $x_0$ of $X_t$, the initial value $y_0$ of $Y_t$ and the volatility $\eta$ of $Y_t$.
In the previous sections we assumed that $\eta\in L(H_w)$.
We recall that the space of bounded linear operators on a Hilbert space is not reflexive.
In fact, it contains a subspace (the diagonal operators with respect to a given orthonormal basis) which is isomorphic to $\ell^{\infty}$, and $\ell^{\infty}$ is not reflexive, see \cite[Theorem 1.11.16, Proposition 1.11.18 and Example 1.11.23]{Megginson}.
Since $L(H_w)$ is not reflexive, it is not a UMD Banach space \cite[p.5]{PV}.
We therefore need to have stronger assumptions on $\eta$ to be able to apply the Malliavin calculus in infinite dimensions and in particular to use the chain rule \cite[Proposition 3.8]{PV} for the Malliavin derivative.
Hence, in this section, we assume that $\eta\in\mH_w\coloneqq L_{HS}(H_w)\subset L(H_w)$.
The option price at time $t=0$ is then a functional on $H_w\times H_w \times \mH_w$ which takes the form
\begin{equation}
\label{Pi_0}
 \Pi_0(x_0,y_0,\eta) = \e^{-r\tau}\E\left[(\Phi\circ\mJ_{x,d})(X_{\tau}(x_0,y_0,\eta))\right] = \E\left[\Psi(X_{\tau}(x_0,y_0,\eta))\right],
\end{equation}
where the notation $X_{\tau}(x_0,y_0,\eta)$ means that we regard the random variable $X_\tau$ as a function of $x_0$, $y_0$ and $\eta$, and to ease the notation, we have introduced the functional $\Psi$ on $H_w$ as $\Psi\coloneqq\e^{-r\tau}(\Phi\circ\mJ_{x,d})$.

We are going to consider $\Pi_0$ as a function of one of the parameters keeping the two others fixed.
For this purpose we will use the notation $\Pi_0^{y_0,\eta}$, $\Pi_0^{x_0,\eta}$ and $\Pi_0^{x_0,y_0}$, where the variables in superscript are the ones we keep fixed.
The directional derivative $\partial_h\Pi_0^{y_0,\eta}(x_0)$ of $\Pi_0^{y_0,\eta}$ at $x_0\in H_w$ in direction $h\in H_w$ is defined as
\begin{equation}\label{directional-derivative}
\partial_{h} \Pi_0^{y_0,\eta}(x_0) \coloneqq \frac{d}{d\varepsilon} \Pi_0^{y_0,\eta}(x_0+\varepsilon h) \bigg{|}_{\varepsilon=0} = \lim_{\varepsilon\rightarrow 0} \frac{\Pi_0^{y_0,\eta}(x_0 + \varepsilon h) - \Pi_0^{y_0,\eta}(x_0)}{\varepsilon}, \quad h\in H_w.  
\end{equation}
This will be our interpretation of the Greek delta.
The  directional derivative $\partial_h\Pi_0^{x_0,\eta}(y_0)$ is defined similarly. 
The directional derivative of $\Pi_0^{x_0,y_0}$ at $\eta\in \mH_w$ in direction $\zeta\in \mH_w$ is defined as
\begin{equation}\label{directional-derivative-V3}
 \partial_{\zeta} \Pi_0^{x_0,y_0}(\eta) \coloneqq \frac{d}{d\varepsilon} \Pi_0^{x_0,y_0}(\eta+\varepsilon \zeta) \bigg{|}_{\varepsilon=0} = \lim_{\varepsilon\rightarrow 0} \frac{\Pi_0^{x_0,y_0}(\eta + \varepsilon \zeta) - \Pi_0^{x_0,y_0}(\eta)}{\varepsilon}, \quad \zeta\in \mH_w.  
\end{equation}
Before proceeding further, we briefly review some elements of Malliavin calculus on Hilbert spaces.
We refer to \cite{N} and \cite{DPO} for an introduction to Malliavin calculus, and to \cite{CT} and \cite{PV} for more details on Malliavin calculus on Hilbert and Banach spaces.

\subsection{Some elements of Malliavin calculus}
Let us consider an isonormal process $\bW$ on a filtered probability space $(\Omega,\mF,\bF, \bP)$ and some Hilbert space $H$, where $\bF=\{\mF_t\}_{0\leq t\leq T}$ is the filtration generated by $\bW$ and $\mF=\mF_T$.
For all $h_1,\ldots,h_n\in H$, $\bW(h_1),\ldots,\bW(h_n)$ are jointly normally distributed real-valued random variables with mean zero and $\E[\bW(h_i)\bW(h_j)] = \<h_i,h_j\>_H$.

Let $E$ be another Hilbert space, which is assumed to be separable. 
An $E$-valued random variable $F$ is called \emph{smooth} if there exists $h_1,\dots,h_n\in H$ such that $F$ can be written on the form $F = f(\bW(h_1),\dots,\bW(h_n))$, where $f:\R^n\rightarrow E$ is infinitely differentiable and all derivatives are polynomially bounded.
The set of smooth $E$-valued random variables is dense in $L^p(\Omega;E)$ for $1\leq p<\infty$.
The Malliavin derivative $\sD F$ of a smooth $E$-valued random variable $F$ is a random variable taking values in $E\x H$, and it is defined as 
\begin{equation*}
\sD F \coloneqq \sum_{i=1}^n \frac{\partial f}{\partial x_i}(\bW(h_1),\ldots,\bW(h_n)) \x h_i.
\end{equation*}
The space $E\x H$ can be identified with the space $L_{HS}(H,E)$ of Hilbert-Schmidt operators from $H$ to $E$.
In the special case when $E=\R$, $L_{HS}(H,\R)$ can be identified with $H$.
The Malliavin derivative of a real-valued random variable is therefore an $H$-valued random variable.

The Malliavin derivative is closable as an unbounded operator from $L^p(\Omega;E)$ to $L^p(\Omega; E\x H)$, and the closure will also be denoted by $\sD$.
The domain of the closure is denoted by $\bD^{1,p}(E)$, and it becomes a Banach space if we endow it with the following norm:
\begin{equation*}
\| F \|_{\bD^{1,p}(E)} \coloneqq \left(\| F \|_{L^p(\Omega;E)}^p + \| \sD F \|_{L^p(\Omega;E\x H)}^p\right)^{1/p}.
\end{equation*}
The space $\bD^{1,2}(E)$ is a Hilbert space.

We will need the following chain rule. We state it here for Hilbert spaces, but it is also valid in the more general case of UMD Banach spaces, see \cite[Proposition 3.8]{PV}.

\begin{lemma}[Chain rule]\label{chain-rule}
Let $E_1$ and $E_2$ be Hilbert spaces, and suppose $\varphi:E_1 \rightarrow E_2$ is Fr\'{e}chet differentiable with a continuous and bounded derivative. If $F\in \bD^{1,p}(E_1)$, then $\varphi(F)\in \bD^{1,p}(E_2)$ and
\begin{equation*}
\sD(\varphi(F)) = D\varphi(F) \circ \sD F.
\end{equation*}
\end{lemma}

\begin{proof}
See \cite[Proposition 3.8]{PV}.
\end{proof}

The adjoint operator of the Malliavin derivative is denoted by $\delta$ and is often called the \emph{Skorohod integral}.
We remark that this is not to be confused with the evaluation functional for which we used the same notation.
The domain of $\delta$ is the set of random variables $u\in L^2(\Omega;E\x H)$ for which there exists a constant $C\geq 0$ such that
\begin{equation*}
\left|\E\left[\<\sD F,u\>_{E\x H}\right]\right| \leq C \left(\E\left[\|F\|_E^2\right]\right)^{1/2},
\end{equation*}
for all $F\in \bD^{1,2}(E)$.
For all  $u\in\text{dom}(\delta)$ and $F\in \bD^{1,2}(E)$, the following relation holds
\begin{equation}\label{duality-formula}
\E\left[\<\sD F,u\>_{E\x H}\right] = \E\left[\<F,\delta(u)\>_E\right].
\end{equation}

The following lemmas are stated for the case $E=\R$, since we will only need them in this setting. 
In this case, $\delta$ is an unbounded operator from $L^2(\Omega;H)$ to $L^2(\Omega)$.
The first lemma gives a formula for the Skorohod integral of the product of a real-valued random variable and an $H$-valued random variable.
For a more general formulation of this lemma, see \cite[Lemma 4.9]{PV}.  

\begin{lemma}[Integration by parts]\label{multiplying by scalar}
Let $F\in \bD^{1,2}(\R)$ and $u\in \text{dom}\,\delta$ such that $Fu\in L^2(\Omega;H)$. Then $Fu\in \text{dom}\,\delta$ and the following holds:
\begin{equation*}
\delta(Fu) = F\delta(u) -\<\sD F,u\>_H.
\end{equation*}
\end{lemma}

\begin{proof}
See \cite[Proposition 1.3.3]{N}.
\end{proof}

\begin{lemma}\label{estimate}
Let $u\in\bD^{1,2}(H)$. Then $u\in \text{dom}\,\delta$ and the following estimate holds:
\begin{equation*}
\|\delta(u)\|_{L^2(\Omega)}^2 \leq \|u\|_{\bD^{1,2}(H)}^2 = \| u \|_{L^2(\Omega;H)}^2 + \| \sD u \|_{L^2(\Omega;H\x H)}^2.
\end{equation*}
\end{lemma}

\begin{proof}
See \cite[Proposition 1.3.1]{N}.
\end{proof}

\subsection{Sensitivity of $\Pi_0$ with respect to $x_0$, $y_0$ and $\eta$}
We now return to the computation of the directional derivatives of $\Pi_0$ defined in \eqref{directional-derivative} and \eqref{directional-derivative-V3}.
In the framework of the previous subsection, we choose $H$ to be the Filipovi\'{c} space $H_w$, and as the Hilbert space $E$ we will use either $\R$, $H_w$ or $\mH_w$, depending on the context.
Let $(\Omega^{B,W},\mF^{B,W},\bP^{B,W})$ denote the probability space on which $B$ and $W$ are defined.
To apply the Malliavin calculus, we introduce an isonormal Gaussian process $\bW$ on a filtered probability space $(\Omega^{\bW}, \mF^{\bW}, \bF^{\bW}, \bP^{\bW})$ and $H_w$, where $\bP^{\bW}$ is independent of $\bP^{B,W}$, $\bF^{\bW}=\{\mF^{\bW}_t\}_{0\leq t\leq T}$ is the filtration generated by $\bW$, $\mF^{\bW}=\mF_T^{\bW}$.
Note that $\bW$ is independent of $B$ and $W$.
For a random variable $F$, we can write $F(\omega) = F(\omega^{B,W},\omega^{\bW})$.
When we talk about the Malliavin derivative of $F$, we mean the Malliavin derivative of $F$ with respect to $\bW$.

For fixed $\omega^{B,W}\in\Omega^{B,W}$ and $\omega^{\bW}\in\Omega^{\bW}$, we consider $X_{\tau}(\omega^{B,W},\omega^{\bW}):H_w\times H_w \times \mH_w\rightarrow H_w$ as a function of the initial value $x_0$ of $X_{\tau}$, the initial value $y_0$ of $Y_t$ and the volatility $\eta$ of $Y_t$. 
Recall that
\begin{equation*}
X_{\tau}(x_0,y_0,\eta) = \mS_\tau x_0 + \int_0^\tau \mS_{\tau-s}\left(Z_s\x \left(\mathcal{U}_s y_0 + \int_0^s \mathcal{U}_{s-u}\eta\, dW_u\right)\right)\,dB_s \quad \omega-a.e,
\end{equation*}
where $\mS_t$ is the right-shift operator.
Following the notation introduced previously, we let $X_{\tau}^{y_0,\eta}$, $X_{\tau}^{x_0,\eta}$ and $X_{\tau}^{x_0,y_0}$ denote $X_{\tau}$ as a function of $x_0$, $y_0$ and $\eta$ respectively, keeping the two other variables fixed.
The Fr\'{e}chet derivatives of $X_{\tau}$ are given in the following lemma.

\begin{lemma}[Fr\'{e}chet derivatives of $X_{\tau}$]
The Fr\'{e}chet derivatives of $X_{\tau}$ with respect to $x_0$, $y_0$ and $\eta$ are given by:
\begin{itemize}
\item[(i)] $DX_{\tau}^{y_0,\eta}(x_0)(h) = \mS_\tau(h)$, \quad $h\in H_w$.
\item[(ii)] $DX_{\tau}^{x_0,\eta}(y_0)(h) =  \int_0^\tau \mS_{\tau-s}\left(Z_s\x \mathcal{U}_s h\right)\,dB_s$, \quad $\omega-a.e.$, $h\in H_w$.
\item[(iii)] $DX_{\tau}^{x_0,y_0}(\eta)(\zeta) =  \int_0^\tau \mS_{\tau-s}\left(Z_s\x \int_0^s \mathcal{U}_{s-u}\zeta\, dW_u\right)\,dB_s$, \quad $\omega-a.e.$, $\zeta\in \mH_w$.
\end{itemize}
\end{lemma}

\begin{proof}
$ $
\begin{itemize}
\item[(i)] The Fr\'{e}chet derivative $DX_{\tau}^{y_0,\eta}(x_0)$ of $X_{\tau}^{y_0,\eta}$ at $x_0\in H_w$ is defined as the bounded linear operator $L:H_w\rightarrow H_w$ such that
\begin{equation*}
\lim_{\|h\|_w\rightarrow 0} \frac{\|X_{\tau}^{y_0,\eta}(x_0+h)-X_{\tau}^{y_0,\eta}(x_0)-Lh\|_w}{\|h\|_w} = 0\quad \omega-a.e.
\end{equation*}
Since $X_{\tau}^{y_0,\eta}(x_0+h)-X_{\tau}^{y_0,\eta}(x_0) = \mS_\tau(h)$, the result follows.
\item[(ii)] 
Since 
\begin{equation*}
X_{\tau}^{x_0,\eta}(y_0+h)-X_{\tau}^{x_0,\eta}(y_0) = \int_0^\tau \mS_{\tau-s}\left(Z_s\x \mathcal{U}_s h\right)\,dB_s\quad \omega-a.e.,
\end{equation*}
for $y_0,h\in H_w$, the result follows from the definition of the Fr\'{e}chet derivative.
\item[(iii)] 
Since 
\begin{equation*}
X_{\tau}^{x_0,y_0}(\eta+\zeta)-X_{\tau}^{x_0,y_0}(\eta) = \int_0^\tau \mS_{\tau-s}\left(Z_s\x \int_0^s \mathcal{U}_{s-u}\zeta\, dW_u\right)\,dB_s\quad \omega-a.e.,
\end{equation*}
for $\eta,\zeta\in\mH_w$, the result follows from the definition of the Fr\'{e}chet derivative.
\end{itemize}
\end{proof}

Assuming that $\Psi$ (see \eqref{Pi_0}) is Fr\'{e}chet differentiable, the composition $\Psi\circ X_{\tau}(\omega):H_w\times H_w \times \mH_w\rightarrow\R$ is also Fr\'{e}chet differentiable with respect to each of the variables.
For $h\in H_w$, it holds that
\begin{equation*}
D(\Psi\circ X_{\tau}^{y_0,\eta})(x_0)(h) = \partial_h(\Psi\circ X_{\tau}^{y_0,\eta})(x_0) = \lim_{\varepsilon\rightarrow 0} \frac{\Psi(X_{\tau}^{y_0,\eta}(x_0 + \varepsilon h)) - \Psi(X_{\tau}^{y_0,\eta}(x_0))}{\varepsilon} \quad\omega-a.e.
\end{equation*}
Taking expectations, we have that
\begin{equation*}
\E\left[D(\Psi\circ X_{\tau}^{y_0,\eta})(x_0)(h)\right] = \E\left[\lim_{\varepsilon\rightarrow 0} \frac{\Psi(X_{\tau}^{y_0,\eta}(x_0 + \varepsilon h)) - \Psi(X_{\tau}^{y_0,\eta}(x_0))}{\varepsilon}\right].
\end{equation*}
From \eqref{directional-derivative} we recall that 
\begin{equation*}
\partial_h \Pi_0^{y_0,\eta}(x_0) = \lim_{\varepsilon\rightarrow 0} \frac{\Pi_0^{y_0,\eta}(x_0 + \varepsilon h) - \Pi_0^{y_0,\eta}(x_0)}{\varepsilon} = \lim_{\varepsilon\rightarrow 0} \E\left[\frac{\Psi(X_{\tau}^{y_0,\eta}(x_0 + \varepsilon h)) - \Psi(X_{\tau}^{y_0,\eta}(x_0))}{\varepsilon}\right]. 
\end{equation*}
The following lemma shows that we can move the limit inside the expectation, and hence express the directional derivatives of $\Pi_0$ in terms of the Fr\'{e}chet derivatives of $\Psi\circ X_{\tau}$.

\begin{lemma}\label{step1}
Assume that $\Psi$ is Fr\'{e}chet differentiable. Then the following holds:
\begin{itemize}
\item[(i)] $\partial_h\Pi_0^{y_0,\eta}(x_0) = \E\left[D(\Psi\circ X_{\tau}^{y_0,\eta})(x_0)(h)\right]$, \quad $h\in H_w$.
\item[(ii)] $\partial_h\Pi_0^{x_0,\eta}(y_0) = \E\left[D(\Psi\circ X_{\tau}^{x_0,\eta})(y_0)(h)\right]$, \quad $h\in H_w$.
\item[(iii)] $\partial_{\zeta}\Pi_0^{x_0,y_0}(\eta) = \E\left[D(\Psi\circ X_{\tau}^{x_0,y_0})(\eta)(\zeta)\right]$, \quad $\zeta\in\mH_w$.
\end{itemize}
\end{lemma}

\begin{proof}
Let $L_{\Psi}$ denote the Lipschitz constant of $\Psi$.

(i):
For each $\varepsilon\in\R\backslash\{0\}$ and for almost all $\omega\in\Omega$, we have that

\begin{align*}
\left|\frac{\Psi(X_{\tau}^{y_0,\eta}(x_0 + \varepsilon h)) - \Psi(X_{\tau}^{y_0,\eta}(x_0))}{\varepsilon}\right|
&\leq \frac{L_{\Psi}\|X_{\tau}^{y_0,\eta}(x_0 + \varepsilon h) - X_{\tau}^{y_0,\eta}(x_0)\|_w}{|\varepsilon|}\\
&= \frac{L_{\Psi}\|\mS_\tau(x_0 + \varepsilon h) - \mS_\tau(x_0)\|_w}{|\varepsilon|}\\
&= \frac{L_{\Psi}\|\mS_\tau(\varepsilon h)\|_w}{|\varepsilon|}\\
&\leq L_{\Psi}\|\mS_\tau(h)\|_w\\
&\leq L_{\Psi}\|\mS_\tau\|_{\text{op}}\|h\|_w\\
&\leq L_{\Psi}C\|h\|_w,
\end{align*}
where the constant $C$ is given in Lemma \ref{shift-op-bound}.
The desired result then follows by the bounded convergence theorem.

(ii):
For each $\varepsilon\in\R\backslash\{0\}$, we have that
\begin{equation*}
\left|\frac{\Psi(X_{\tau}^{x_0,\eta}(y_0 + \varepsilon h)) - \Psi(X_{\tau}^{x_0,\eta}(y_0))}{\varepsilon}\right|
\leq \frac{L_{\Psi}\|X_{\tau}^{x_0,\eta}(y_0 + \varepsilon h) - X_{\tau}^{x_0,\eta}(y_0)\|_w}{|\varepsilon|}\eqqcolon A_{\varepsilon},
\end{equation*}
which means that
\begin{equation*}
-A_{\varepsilon} \leq \frac{\Psi(X_{\tau}^{x_0,\eta}(y_0 + \varepsilon h)) - \Psi(X_{\tau}^{x_0,\eta}(y_0))}{\varepsilon} \leq A_{\varepsilon}.
\end{equation*}
Since
\begin{equation*}
DX_{\tau}^{x_0,\eta}(y_0)(h) = \lim_{\varepsilon\to 0} \frac{X_{\tau}^{x_0,\eta}(y_0 + \varepsilon h) - X_{\tau}^{x_0,\eta}(y_0)}{\varepsilon}\quad\omega-a.e.,
\end{equation*}
it follows by the continuity of the norm $\|\cdot\|_w$ that
\begin{equation*}
\|DX_{\tau}^{x_0,\eta}(y_0)(h)\|_w = \lim_{\varepsilon\to 0} \frac{\|X_{\tau}^{x_0,\eta}(y_0 + \varepsilon h) - X_{\tau}^{x_0,\eta}(y_0)\|_w}{|\varepsilon|}.
\end{equation*}
This gives
\begin{equation*}
\lim_{\varepsilon\to 0} A_{\varepsilon} = L_{\Psi} \|DX_{\tau}^{x_0,\eta}(y_0)(h)\|_w \eqqcolon A\quad\omega-a.e.
\end{equation*}
If we can show that $A_{\varepsilon}\rightarrow A$ in $L^1(\Omega)$ when $\varepsilon\rightarrow 0$, the desired result will follow by Pratt's lemma, see \cite{Pratt}.
To show convergence in $L^1(\Omega)$ we will use Vitali's theorem.
We first show that $\|A_{\varepsilon}\|_{L^2(\Omega)}\leq M$ for all $\varepsilon>0$ for some constant $M$.
By the It\^{o} isometry and Lemma \ref{shift-op-bound} we have that
\begin{align*}
\E\left[|A_{\varepsilon}|^2\right]
&= \E\left[\frac{L_{\Psi}^2\|X_{\tau}^{x_0,\eta}(y_0 + \varepsilon h) - X_{\tau}^{x_0,\eta}(y_0)\|_w^2}{\varepsilon^2}\right] \\
&= \frac{L_{\Psi}^2}{\varepsilon^2}\E\left[\bigg\|\int_0^\tau \mS_{\tau-s}\left(Z_s\x \mathcal{U}_s (\varepsilon h)\right)\,dB_s\bigg\|_w^2\right]\\
&= L_{\Psi}^2\E\left[\bigg\|\int_0^\tau \mS_{\tau-s}\left(Z_s\x \mathcal{U}_s h\right)\,dB_s\bigg\|_w^2\right]\\
&= L_{\Psi}^2\E\left[\int_0^\tau \| \mS_{\tau-s}\left(Z_s\x \mathcal{U}_s h\right)Q_B^{1/2}\|_{\mH_w}^2\,ds\right]\\
&\leq L_{\Psi}^2\E\left[\int_0^\tau  \|\mS_{\tau-s}\|_{\text{op}}^2 \|\left(Z_s\x \mathcal{U}_s h\right)Q_B^{1/2}\|_{\mH_w}^2\,ds\right]\\
&\leq L_{\Psi}^2\E\left[\int_0^\tau  C^2 \|\left(Z_s\x \mathcal{U}_s h\right)Q_B^{1/2}\|_{\mH_w}^2\,ds\right],
\end{align*}
since $\|\mS_t\|_{\text{op}} \leq C$ where the constant $C$ is given in Lemma \ref{shift-op-bound}.
Let $\{e_n\}_{n\in\N}$ be an orthonormal basis of $H_w$.
By Parseval's identity we have that
\begin{align*}
\|\left(Z_s\x \mathcal{U}_s h\right)Q_B^{1/2}\|_{\mH_w}^2 &= \sum_{n=1}^{\infty} \|\left(Z_s\x \mU_s h\right)Q_B^{1/2} e_n\|_w^2 \\
&= \sum_{n=1}^{\infty} \|\<Z_s,Q_B^{1/2}e_n\>_w \,\mU_s h\|_w^2 \\
&= \|\mU_s h\|_w^2 \sum_{n=1}^{\infty} |\<Z_s,Q_B^{1/2}e_n\>_w|^2 \\
&= \|\mU_s h\|_w^2 \sum_{n=1}^{\infty} |\<Q_B^{1/2}Z_s,e_n\>_w|^2 \\
&= \|\mU_s h\|_w^2 \|Q_B^{1/2}Z_s\|_w^2.
\end{align*}
By the Hille-Yosida theorem we have that $\|\mU_{t}\|_{\text{op}} \leq K\e^{kt}$ for some constants $K$ and $k$.
Hence
\begin{equation*}
\|\mU_s h\|_w^2 \leq \|\mU_{s}\|_{\text{op}}^2 \|h\|_w^2 \leq K^2\e^{2ks} \|h\|_w^2.
\end{equation*}

Let $\{v_n\}_{n\in\N}$ be the orthonormal basis of $H_w$ consisting of eigenvectors of $Q_B$ with corresponding eigenvalues $\{\lambda_n\}_{n\in\N}$.
We know that such a basis exists since $Q_B$ is a symmetric, positive definite trace class operator.
By Parseval's identity and Cauchy-Schwarz inequality, we find that
\begin{align*}
\|Q_B^{1/2}Z_s\|_w^2 &= \<Q_B^{1/2}Z_s,Q_B^{1/2}Z_s\>_w = \<Q_B Z_s,Z_s\>_w = \sum_{n=1}^{\infty} \<Z_s,Q_B v_n\>_w \<Z_s,v_n\>_w \\
&= \sum_{n=1}^{\infty} \lambda_n \<Z_s,v_n\>_w^2 \leq \sum_{n=1}^{\infty} \lambda_n \|Z_s\|_w^2 \|v_n\|_w^2 = \sum_{n=1}^{\infty} \lambda_n = \text{Tr}(Q_B).
\end{align*}
Hence we have that
\begin{equation*}
\E\left[|A_{\varepsilon}|^2\right] \leq L_{\Psi}^2 C^2 K^2\text{Tr}(Q_B) \|h\|_w^2 \int_0^\tau \e^{2c(\tau-s)}e^{2ks}\,ds =: M.
\end{equation*}
For a fixed $\delta>0$, it then holds by H\"{o}lder's inequality and Markov's inequality that
\begin{align*}
\lim_{N\rightarrow\infty} \sup_{|\varepsilon|<\delta} \E\left[|A_{\varepsilon}| 1_{\{|A_{\varepsilon}|>N\}}\right] 
&\leq \lim_{N\rightarrow\infty} \sup_{|\varepsilon|<\delta} \E\left[|A_{\varepsilon}|^2\right]^{1/2} \E\left[1_{\{|A_{\varepsilon}|>N\}}\right]^{1/2}\\
&= \lim_{N\rightarrow\infty} \sup_{|\varepsilon|<\delta} \|A_{\varepsilon}\|_{L^2(\Omega)} \sqrt{\bP(|A_{\varepsilon}|>N)}\\
&\leq \lim_{N\rightarrow\infty} \sup_{|\varepsilon|<\delta} \|A_{\varepsilon}\|_{L^2(\Omega)} \frac{\|A_{\varepsilon}\|_{L^2(\Omega)}}{N}\\
&= \lim_{N\rightarrow\infty} \sup_{|\varepsilon|<\delta} \frac{\|A_{\varepsilon}\|_{L^2(\Omega)}^2}{N}\\
&\leq \lim_{N\rightarrow\infty} \frac{M^2}{N}\\
&=0.
\end{align*}
By Vitali's theorem it follows that $\E[|A_{\varepsilon}-A|]\rightarrow 0$ when $\varepsilon\rightarrow 0$, and by Pratt's lemma it then follows that
\begin{equation*}
\partial_h V_2(y_0) = \E\left[D(\Psi\circ X_{\tau}^{x_0,\eta})(y_0)(h)\right].
\end{equation*}

(iii):
This proof follows the same approach as the proof of (ii).
For each $\varepsilon\in\R\backslash\{0\}$, we have that
\begin{equation*}
\left|\frac{\Psi(X_{\tau}^{x_0,y_0}(\eta + \varepsilon \zeta)) - \Psi(X_{\tau}^{x_0,y_0}(\eta))}{\varepsilon}\right|
\leq \frac{L_{\Psi}\|X_{\tau}^{x_0,y_0}(\eta + \varepsilon \zeta) - X_{\tau}^{x_0,y_0}(\eta)\|_w}{|\varepsilon|}\eqqcolon A_{\varepsilon},
\end{equation*}
which gives
\begin{equation*}
-A_{\varepsilon} \leq \frac{\Psi(X_{\tau}^{x_0,y_0}(\eta + \varepsilon \zeta)) - \Psi(X_{\tau}^{x_0,y_0}(\eta))}{\varepsilon} \leq A_{\varepsilon}.
\end{equation*}
By the continuity of the norm $\|\cdot\|_w$ it follows that
 \begin{equation*}
\lim_{\varepsilon\to 0} A_{\varepsilon} = L_{\Psi} \|DX_{\tau}^{x_0,y_0}(\eta)(\zeta)\|_w \eqqcolon A\quad\omega-a.e.
\end{equation*}
We show that $\|A_{\varepsilon}\|_{L^2(\Omega)}\leq M$ for all $\varepsilon>0$ for some constant $M$.
By the It\^{o} isometry, the Hille-Yosida thorem and Fubini we have that
\begin{align*}
\E\left[|A_{\varepsilon}|^2\right]
&= \E\left[\frac{L_{\Psi}^2\|X_{\tau}^{x_0,y_0}(\eta + \varepsilon \zeta) - X_{\tau}^{x_0,y_0}(\eta)\|_w^2}{\varepsilon^2}\right] \\
&= \frac{L_{\Psi}^2}{\varepsilon^2}\E\left[\bigg\|\int_0^\tau \mS_{\tau-s}\left(Z_s\x \int_0^s \mathcal{U}_{s-u}\varepsilon\zeta\, dW_u\right)\,dB_s\bigg\|_w^2\right]\\
&= L_{\Psi}^2\E\left[\bigg\|\int_0^\tau \mS_{\tau-s}\left(Z_s\x \int_0^s \mathcal{U}_{s-u}\zeta\, dW_u\right)\,dB_s\bigg\|_w^2\right]\\
&= L_{\Psi}^2\E\left[\int_0^\tau \left\| \mS_{\tau-s}\left(Z_s\x \int_0^s \mathcal{U}_{s-u}\zeta\, dW_u\right)Q_B^{1/2}\right\|_{\mH_w}^2\,ds\right]\\
&\leq L_{\Psi}^2\E\left[\int_0^\tau  \|\mS_{\tau-s}\|_{\text{op}}^2 \left\|\left(Z_s\x \int_0^s \mathcal{U}_{s-u}\zeta\, dW_u\right)Q_B^{1/2}\right\|_{\mH_w}^2\,ds\right]\\
&\leq L_{\Psi}^2\E\left[\int_0^\tau  C^2\e^{2c(\tau-s)} \left\|\left(Z_s\x \int_0^s \mathcal{U}_{s-u}\zeta\, dW_u\right)Q_B^{1/2}\right\|_{\mH_w}^2\,ds\right]\\
&= L_{\Psi}^2 C^2 \int_0^\tau  \e^{2c(\tau-s)} \E\left[\left\|\left(Z_s\x \int_0^s \mathcal{U}_{s-u}\zeta\, dW_u\right)Q_B^{1/2}\right\|_{\mH_w}^2\right]\,ds.
\end{align*}
Let $\{e_n\}_{n\in\N}$ be an orthonormal basis of $H_w$.
By Parseval's identity we have that
\begin{align*}
\left\|\left(Z_s\x \int_0^s \mathcal{U}_{s-u}\zeta\, dW_u\right)Q_B^{1/2}\right\|_{\mH_w}^2 &= \sum_{n=1}^{\infty} \left\|\left(Z_s\x \int_0^s \mathcal{U}_{s-u}\zeta\, dW_u\right)Q_B^{1/2} e_n\right\|_w^2 \\
&= \sum_{n=1}^{\infty} \left\|\<Z_s,Q_B^{1/2}e_n\>_w \,\int_0^s \mathcal{U}_{s-u}\zeta\, dW_u\right\|_w^2 \\
&= \left\|\int_0^s \mathcal{U}_{s-u}\zeta\, dW_u\right\|_w^2 \sum_{n=1}^{\infty} |\<Z_s,Q_B^{1/2}e_n\>_w|^2 \\
&= \left\|\int_0^s \mathcal{U}_{s-u}\zeta\, dW_u\right\|_w^2 \sum_{n=1}^{\infty} |\<Q_B^{1/2}Z_s,e_n\>_w|^2 \\
&= \left\|\int_0^s \mathcal{U}_{s-u}\zeta\, dW_u\right\|_w^2 \left\|Q_B^{1/2}Z_s\right\|_w^2.
\end{align*}
By the It\^{o} isometry and the Hille-Yosida theorem we have that
\begin{align*}
\E\left[\left\|\int_0^s \mathcal{U}_{s-u}\zeta\, dW_u\right\|_w^2\right] &= \E\left[\int_0^s \left\|\mathcal{U}_{s-u}\zeta Q_W^{1/2}\right\|_{\mH_w}^2 du\right]\\
&\leq \E\left[\int_0^s \|\mathcal{U}_{s-u}\|_{\text{op}}^2 \|\zeta\|_{\text{op}}^2 \|Q_W^{1/2}\|_{\mH_w}^2 du\right]\\
&= \text{Tr}(Q_W) \|\zeta\|_{\text{op}}^2 \int_0^s \|\mathcal{U}_{s-u}\|_{\text{op}}^2 du\\
&\leq \text{Tr}(Q_W) \|\zeta\|_{\text{op}}^2 K^2 \int_0^s e^{2k(s-u)} du.
\end{align*}
Since $\|Q_B^{1/2}Z_s\|_w^2 \leq \text{Tr}(Q_B)$, it follows that
\begin{equation*}
\E\left[|A_{\varepsilon}|^2\right] \leq L_{\Psi}^2 C^2 K^2\text{Tr}(Q_B)\text{Tr}(Q_W) \|\zeta\|_{\text{op}}^2 \int_0^\tau \int_0^s \e^{2c(\tau-s)}e^{2k(s-u)}\,duds =: M.
\end{equation*}
We have now shown that there exists a constant $M$ such that $\|A_{\varepsilon}\|_{L^2(\Omega)}\leq M$ for all $\varepsilon>0$.
By the same argument as in the proof of (ii), it then follows from Vitali's theorem that $A_{\varepsilon}\rightarrow A$ in $L^1(\Omega)$ when $\varepsilon\rightarrow 0$.
The desired result then follows by Pratt's lemma.
\end{proof}

The following lemma gives expressions for the expectation of the Fr\'{e}chet derivatives of $\Psi\circ X_{\tau}$.
The trick we use to compute these Fr\'{e}chet derivatives is to randomize the parameter that we want to differentiate with respect to with an $H_w$-valued noise independent of $B$ and $W$.
By applying the chain rule we can then express the Malliavin derivative with respect to this noise in terms of the Fr\'{e}chet derivative that we want to compute.

\begin{lemma}\label{technical-lemma}
Let $\xi$ be a real-valued random variable on $(\Omega^{\bW}, \mF^{\bW}, \bP^{\bW})$.
Assume that $\xi\in \bD^{1,2}(\R^+)$ and that $\sD\xi(\omega,x)\neq 0$ for all $x\in\R^+$ and almost all $\omega\in\Omega$.
Let $x\in\R^+$, and assume that $\frac{\xi}{\sD\xi(x)} h_x$ is Skorohod integrable and that the evaluation of the Skorohod integral at $\lambda=\frac{1}{\xi}$ is well defined.
Assume also that the Skorohod integrals below and their evaluations at $\lambda=\frac{1}{\xi}$ are well defined.
Then it holds that
\begin{itemize}
\item[(i)] $\E\left[D(\Psi\circ X_{\tau}^{y_0,\eta})(x_0)(h)\right] = -\E\left[\left\{\delta\left(\Psi(X_{\tau}^{y_0,\eta}(x_0-h+\lambda\xi h)) \frac{\xi}{\sD\xi(x)}h_x\right)\right\}\bigg{|}_{\lambda=\frac{1}{\xi}}\right]$,
\item[(ii)] $\E\left[D(\Psi\circ X_{\tau}^{x_0,\eta})(y_0)(h)\right] = -\E\left[\left\{\delta\left(\Psi(X_{\tau}^{y_0,\eta}(y_0-h+\lambda\xi h)) \frac{\xi}{\sD\xi(x)}h_x\right)\right\}\bigg{|}_{\lambda=\frac{1}{\xi}}\right]$,
\item[(iii)] $\E\left[D(\Psi\circ X_{\tau}^{x_0,y_0})(\eta)(\zeta)\right] = -\E\left[\left\{\delta\left(\Psi(X_{\tau}^{x_0,y_0}(\eta-\zeta+\lambda\xi \zeta)) \frac{\xi}{\sD\xi(x)}h_x\right)\right\}\bigg{|}_{\lambda=\frac{1}{\xi}}\right]$.
\end{itemize}
\end{lemma}

\begin{proof}
We only prove (i) here. The proofs of (ii) and (iii) follow the same approach.
For (iii), note that we can still use the chain rule in Theorem \ref{chain-rule} since we assumed $\eta\in \mH_w$.
For a general $\eta\in L(H_w)$ we can not use the chain rule \cite[Proposition 3.8]{PV} since $L(H_w)$ is not reflexive and hence not a UMD Banach space.

Define first a function $\theta:\R^+\rightarrow H_w$ by $\theta(y) = x_0-h+\lambda y h$ for some $\lambda\in\R$ and $h\in H_w$.
The randomized initial path is then defined as $X_0\coloneqq \theta(\xi)$.
The Fr\'{e}chet derivative $D\theta(y)$ of $\theta$ at $y$ is the bounded linear operator $D\theta(y):\R^+\rightarrow H_w$ given by $D\theta(y)(x) = \lambda x h$.
By the chain rule in Lemma \ref{chain-rule} it holds that $\theta(\xi)\in\bD^{1,2}(H_w)$ and 
\begin{equation}\label{chain-rule-eta}
\sD X_0 = \sD(\theta(\xi)) = D\theta(\xi) \circ \sD \xi \qquad\omega-a.e.
\end{equation}
The Malliavin derivative $\sD\xi$ is an $H_w$-valued random variable.
For any $x\in\R^+$ we have that 
\begin{equation}\label{chain-rule-eta2}
\sD X_0(\omega,x) = D\theta(\xi(\omega))(\sD\xi(\omega,x)) = \lambda \sD\xi(\omega,x) h \qquad\omega-a.e.
\end{equation}
By the chain rule in Lemma \ref{chain-rule} it holds that $X_{\tau}^{y_0,\eta}(X_0)\in\bD^{1,2}(H_w)$.
Since the Fr\'{e}chet derivative of $\Psi$ is also bounded and continuous by assumption,
we also have $\Psi(X_{\tau}^{y_0,\eta}(X_0))\in\bD^{1,2}(\R^+)$ by the chain rule.
The chain rule in Lemma \ref{chain-rule} also gives that
\begin{equation}\label{chain-rule-X_t}
\sD(\Psi(X_{\tau}^{y_0,\eta}(X_0))) = D\Psi(X_{\tau}^{y_0,\eta}(X_0)) \circ DX_{\tau}^{y_0,\eta}(X_0) \circ D\theta(\xi) \circ \sD\xi \qquad\omega-a.e.,
\end{equation}
and therefore for any $x\in\R^+$ and for almost all $\omega\in\Omega$ it holds that
\begin{align*}
\sD (\Psi(X_{\tau}^{y_0,\eta}(X_0)))(\omega,x) &= D\Psi(X_{\tau}^{y_0,\eta}(X_0(\omega)))\big(DX_{\tau}^{y_0,\eta}(X_0(\omega))(D\theta(\xi(\omega))(\sD \xi(\omega,x)))\big)\\
&= D\Psi(X_{\tau}^{y_0,\eta}(X_0(\omega)))\big(DX_{\tau}^{y_0,\eta}(X_0(\omega))(\lambda \sD\xi(\omega,x) h)\big)\\
&= \lambda \sD\xi(\omega,x) D\Psi(X_{\tau}^{y_0,\eta}(X_0(\omega)))\big(DX_{\tau}^{y_0,\eta}(X_0(\omega))(h)\big),
\end{align*}
where we used \eqref{chain-rule-eta2} and then factored out the scalar $\lambda\sD\xi(\omega,x)$ since the Fr\'{e}chet derivative is linear.
By assumption, $\sD\xi(\omega,x)\neq 0$ for all $x\in\R^+$ and almost all $\omega\in\Omega$.
Multiplying with $\frac{\xi(\omega)}{\sD\xi(\omega,x)}$ in the above equation, we have that for any $x\in\R^+$ it holds that
\begin{equation*}
\frac{\xi(\omega)}{\sD\xi(\omega,x)} \sD (\Psi(X_{\tau}^{y_0,\eta}(X_0)))(\omega,x) = \lambda\xi(\omega) D\Psi(X_{\tau}^{y_0,\eta}(X_0(\omega)))(DX_{\tau}^{y_0,\eta}(X_0(\omega))(h)) \qquad\omega-a.e.
\end{equation*}
For a fixed pair $(\omega, x)$ we evaluate the above expression at $\lambda=\frac{1}{\xi(\omega)}$, recalling that $X_0 = x_0-h+\lambda\xi h$.
It holds for almost all $\omega\in\Omega$ that
\begin{align} \label{*}
&\left\{\frac{\xi(\omega)}{\sD\xi(\omega,x)} \sD (\Psi(X_{\tau}^{y_0,\eta}(X_0)))(\omega,x)\right\}\bigg{|}_{\lambda=\frac{1}{\xi(\omega)}} \\ \notag
&= \lambda\xi(\omega) D\Psi(X_{\tau}^{y_0,\eta}(x_0-h+\lambda\xi(\omega) h))(DX_{\tau}^{y_0,\eta}(x_0-h+\lambda\xi(\omega) h)(h))\bigg{|}_{\lambda=\frac{1}{\xi(\omega)}} \\ \notag
&= D\Psi(X_{\tau}^{y_0,\eta}(x_0))(DX_{\tau}^{y_0,\eta}(x_0)(h)).
\end{align}

By the chain rule for Fr\'{e}chet derivatives, equation \eqref{*} and the fact that evaluating a function in the Filipovi\'{c} space $H_w$ at $x\in\R^+$ corresponds to taking the inner product with the function $h_x$ defined in \eqref{h_x}, we get that for almost all $\omega\in\Omega$,
\begin{align*}
D(\Psi\circ X_{\tau}^{y_0,\eta})(x_0)(h) &= D\Psi(X_{\tau}^{y_0,\eta}(x_0))(DX_{\tau}^{y_0,\eta}(x_0)(h))\\
&= \left\{\frac{\xi}{\sD\xi(x)} \sD (\Psi(X_{\tau}^{y_0,\eta}(X_0)))(x)\right\}\bigg{|}_{\lambda=\frac{1}{\xi}}\\
&= \left\{\frac{\xi}{\sD\xi(x)}\Big\<\sD (\Psi(X_{\tau}^{y_0,\eta}(x_0-h+\lambda\xi h))),h_x\Big\>_w\right\}\bigg{|}_{\lambda=\frac{1}{\xi}}\\
&= \left\<\sD (\Psi(X_{\tau}^{y_0,\eta}(x_0-h+\lambda\xi h))),\frac{\xi}{\sD\xi(x)} h_x\right\>_w\bigg{|}_{\lambda=\frac{1}{\xi}}.
\end{align*}

Finally we compute the expectation of the Fr\'{e}chet derivative above. Note that since we have an evaluation in the expression above, we can not apply the duality formula \eqref{duality-formula} here.
We apply instead Lemma \ref{multiplying by scalar} with $F = \Psi(X_{\tau}^{y_0,\eta}(X_0))$ and $u = \frac{\xi}{\sD\xi(x)} h_x$, and get that
\begin{align*}
&\E\left[D(\Psi\circ X_{\tau}^{y_0,\eta})(x_0)(h)\right] \\
&= \E\left[\left\<\sD (\Psi(X_{\tau}^{y_0,\eta}(x_0-h+\lambda\xi h))),\frac{\xi}{\sD\xi(x)} h_x\right\>_w\bigg{|}_{\lambda=\frac{1}{\xi}}\right]\\
&= \E\left[\left\{\Psi(X_{\tau}^{y_0,\eta}(x_0-h+\lambda\xi h)) \delta \left(\frac{\xi}{\sD\xi(x)} h_x\right) - \delta\left(\Psi(X_{\tau}^{y_0,\eta}(x_0-h+\lambda\xi h)) \frac{\xi}{\sD\xi(x)} h_x\right)\right\}\bigg{|}_{\lambda=\frac{1}{\xi}}\right]\\
&= \E\left[\Psi(X_{\tau}^{y_0,\eta}(x_0)) \delta \left(\frac{\xi}{\sD\xi(x)} h_x\right) - \left\{\delta\left(\Psi(X_{\tau}^{y_0,\eta}(x_0-h+\lambda\xi h)) \frac{\xi}{\sD\xi(x)} h_x\right)\right\}\bigg{|}_{\lambda=\frac{1}{\xi}}\right]\\
&= -\E\left[\left\{\delta\left(\Psi(X_{\tau}^{y_0,\eta}(x_0-h+\lambda\xi h)) \frac{\xi}{\sD\xi(x)} h_x\right)\right\}\bigg{|}_{\lambda=\frac{1}{\xi}}\right],
\end{align*}
where we in last equality used that $\Psi(X_{\tau}^{y_0,\eta}(x_0))$ and $\delta \left(\frac{\xi}{\sD\xi(x)} h_x\right)$ are independent since $\Psi(X_{\tau}^{y_0,\eta}(x_0))$ is $\mF^B$-measurable and  $\delta \left(\frac{\xi}{\sD\xi(x)} h_x\right)$ is $\mF^{\bW}$-measurable, and $\E\left[\delta \left(\frac{\xi}{\sD\xi(x)} h_x\right)\right]=0$.
\end{proof}



Our final step is to choose a specific $\xi$ which satisfies the assumptions in Lemma \ref{technical-lemma}.

\begin{theorem} \label{thm:sensitivity:x0}
Assume that $\Psi$ is Lipschitz continuous and Fr\'{e}chet differentiable, and that the Fr\'{e}chet derivative of $\Psi$ is bounded and Lipschitz continuous. 
Let $\xi = \exp(\bW(1_{[0,\tau]}))$, and let $x\in\R^+$.
\begin{itemize}
\item[(i)] For all $h\in H_w$ it holds that
\begin{equation*}
\partial_h\Pi_0^{y_0,\eta}(x_0) = -\E\left[\Big\{\delta\left(\Psi(X_{\tau}^{y_0,\eta}(x_0-h+\lambda\xi h)) h_x \right)\Big\}\bigg{|}_{\lambda=\frac{1}{\xi}}\right].
\end{equation*}
\item[(ii)] For all $h\in H_w$ it holds that
\begin{equation*}
\partial_h\Pi_0^{x_0,\eta}(y_0) = -\E\left[\Big\{\delta\left(\Psi(X_{\tau}^{x_0,\eta}(y_0-h+\lambda\xi h)) h_x \right)\Big\}\bigg{|}_{\lambda=\frac{1}{\xi}}\right].
\end{equation*}
\item[(iii)] For all $\zeta\in \mH_w$ it holds that
\begin{equation*}
\partial_h\Pi_0^{x_0,y_0}(\eta) = -\E\left[\Big\{\delta\left(\Psi(X_{\tau}^{x_0,y_0}(\eta-\zeta+\lambda\xi \zeta)) h_x \right)\Big\}\bigg{|}_{\lambda=\frac{1}{\xi}}\right].
\end{equation*}
\end{itemize}
\end{theorem}

\begin{proof}
Choose $\xi = \exp(\bW(1_{[0,T]}))$ for some $T>0$. 
Then $\xi\in \bD^{1,2}(\R^+)$ and $\sD\xi(y) = \xi$ for $y\in[0,T]$.

\noindent\textbf{Proof of (i):}
Let
\begin{equation*}
u(z,\lambda) \coloneqq \Psi(X_{\tau}^{y_0,\eta}(x_0-h+\lambda\xi h)) h_x(z)
\end{equation*}
for $z\in\R^+$ and $\lambda\in\R$.
Since $\Psi(X_{\tau}^{y_0,\eta}(x_0-h+\lambda\xi h))\in\bD^{1,2}(\R)$ and $h_x$ is deterministic, $u(\cdot,\lambda)$ is Skorohod integrable for all $\lambda\in\R$.  
We need to show that the evaluation of the Skorohod integral at $\lambda=\frac{1}{\xi}$ is well defined.
If for all $\Lambda>0$, there exists a $C>0$ such that 
\begin{equation}
\label{kolmogorov-condition}
\E\left[\left|\delta\big(u(\cdot,\lambda_1)-u(\cdot,\lambda_2)\big)\right|^2\right] < C|\lambda_1 - \lambda_2|^2,  
\end{equation}
for all $\lambda_1$, $\lambda_2$ $\in$ supp $(\xi^{-1})$ with $|\lambda_1|$, $|\lambda_2|<\Lambda$,
then the process $\lambda\mapsto\delta(u(\cdot,\lambda))$ has a continuous version by Kolmogorov's continuity theorem, and the evaluation at $\lambda = \frac{1}{\xi}$ is well-defined.

We will now verify that \eqref{kolmogorov-condition} holds.
Choose $\Lambda>0$ and $\lambda_1$, $\lambda_2$ $\in$ supp $(\xi^{-1})$ with $|\lambda_1|$, $|\lambda_2|<\Lambda$.
By Lemma \ref{estimate} we have that
\begin{equation}
\label{1.47}
\E\left[\left|\delta\big(u(\cdot,\lambda_1)-u(\cdot,\lambda_2)\big)\right|^2\right] \leq  \|u(\cdot,\lambda_1)-u(\cdot,\lambda_2)\|_{\bD^{1,2}(H_w)}^2,
\end{equation}
where we recall that 
\begin{equation*}
\|u(\cdot,\lambda_1)-u(\cdot,\lambda_2)\|_{\bD^{1,2}(H_w)}^2 = \| u(\cdot,\lambda_1)-u(\cdot,\lambda_2) \|_{L^2(\Omega;H_w)}^2 + \| \sD \left(u(\cdot,\lambda_1)-u(\cdot,\lambda_2)\right) \|_{L^2(\Omega;H_w\x H_w)}^2.
\end{equation*}
We have that
\begin{align*}
&\|X_{\tau}^{y_0,\eta}(x_0-h+\lambda_1\xi h) - X_{\tau}^{y_0,\eta}(x_0-h+\lambda_2\xi h)\|_{L^2(\Omega;H_w)}^2\\
&= \|\mS_\tau(x_0-h+\lambda_1\xi h) - \mS_\tau(x_0-h+\lambda_2\xi h)\|_{L^2(\Omega;H_w)}^2\\
&= \|(\lambda_1-\lambda_2)\xi\mS_\tau(h) \|_{L^2(\Omega;H_w)}^2\\
&\leq |\lambda_1-\lambda_2|^2 \|\mS_\tau(h) \|_w^2 \E[|\xi|^2]\\
&= \|h \|_w^2  \E[|\xi|^2] |\lambda_1-\lambda_2|^2.
\end{align*}
Since $\Psi$ is Lipscitz with Lipschitz constant $L_{\Psi}$, it follows that
\begin{align*}
&\E\left[\big|\Psi(X_{\tau}^{y_0,\eta}(x_0-h+\lambda_1\xi h)) - \Psi(X_{\tau}^{y_0,\eta}(x_0-h+\lambda_2\xi h))\big|^2\right] \\
&\leq  L_{\Psi}^2 \|X_{\tau}^{y_0,\eta}(x_0-h+\lambda_1\xi h) - X_{\tau}^{y_0,\eta}(x_0-h+\lambda_2\xi h)\|_{L^2(\Omega;H_w)}^2\\
&\leq L_{\Psi}^2  \|h \|_w^2  \E[|\xi|^2] |\lambda_1-\lambda_2|^2.
\end{align*}
Then we have
\begin{align*}
\| u(\cdot,\lambda_1)-u(\cdot,\lambda_2) \|_{L^2(\Omega;H_w)}^2 &= \|\left(\Psi(X_{\tau}^{y_0,\eta}(x_0-h+\lambda_1\xi h)) - \Psi(X_{\tau}^{y_0,\eta}(x_0-h+\lambda_2\xi h)) \right) h_x \|_{L^2(\Omega;H_w)}^2\\
&= \|h_x\|_w^2 \E\left[\big|\Psi(X_{\tau}^{y_0,\eta}(x_0-h+\lambda_1\xi h)) - \Psi(X_{\tau}^{y_0,\eta}(x_0-h+\lambda_2\xi h))\big|^2\right] \\
&\leq L_{\Psi}^2 \|h_x\|_w^2  \|h \|_w^2  \E[|\xi|^2] |\lambda_1-\lambda_2|^2.
\end{align*}
By the chain rule in Lemma \ref{chain-rule} and using that we have chosen $\xi$ such that $\sD\xi(y) = \xi$ for all $y\in[0,T]$, the following holds for all $y\in[0,T]$
\begin{align*}
&\big|\sD\big(\Psi(X_{\tau}^{y_0,\eta}(x_0-h+\lambda_1\xi h)) - \Psi(X_{\tau}^{y_0,\eta}(x_0-h+\lambda_2\xi h))\big)(y)\big|^2 \\
&= \big|\sD\Psi(X_{\tau}^{y_0,\eta}(x_0-h+\lambda_1\xi h))(x) - \sD\Psi(X_{\tau}^{y_0,\eta}(x_0-h+\lambda_2\xi h))(y)\big|^2 \\
&= \big|D\Psi(X_{\tau}^{y_0,\eta}(x_0-h+\lambda_1\xi h))(DX_{\tau}^{y_0,\eta}(x_0-h+\lambda_1\xi h)(\lambda_1\sD\xi(y) h)) \\
&\qquad - D\Psi(X_{\tau}^{y_0,\eta}(x_0-h+ \lambda_2\xi h))(DX_{\tau}^{y_0,\eta}(x_0-h+\lambda_2\xi h)(\lambda_2\sD\xi(y) h))\big|^2 \\
&= \big|\lambda_1\xi D\Psi(X_{\tau}^{y_0,\eta}(x_0-h+\lambda_1\xi h))(DX_{\tau}^{y_0,\eta}(x_0-h+\lambda_1\xi h)(h)) \\
&\quad - \lambda_2\xi D\Psi(X_{\tau}^{y_0,\eta}(x_0-h+\lambda_2\xi h))(DX_{\tau}^{y_0,\eta}(x_0-h+\lambda_2\xi h)(h))\big|^2.
\end{align*}

Next, we use that $A_1 x_1 - A_2 x_2 = (A_1-A_2)x_1 + A_2 (x_1-x_2)$ with $A_i = D\Psi(X_{\tau}^{y_0,\eta}(x_0-h+\lambda_i\xi h))$ and $x_i = \lambda_i\xi DX_{\tau}^{y_0,\eta}(x_0-h+\lambda_i\xi h)(h)$ for $i=1,2$.
We also use the property $|a+b|^2\leq 2|a|^2 + 2|b|^2$, the assumption that $D\Psi$ is Lipschitz continuous with Lipschitz constant $L_{D\Psi}$ and that $DX_{\tau}^{y_0,\eta}(h)(g) = \mS_\tau(g)$.
\begin{align*}
&\E\Big[\big|\sD\big(\Psi(X_{\tau}^{y_0,\eta}(x_0-h+\lambda_1\xi h)) - \Psi(X_{\tau}^{y_0,\eta}(x_0-h+\lambda_2\xi h))\big)(y)\big|^2\Big] \\
&= \E\Big[ \big|\big(D\Psi(X_{\tau}^{y_0,\eta}(x_0-h+\lambda_1\xi h))-D\Psi(X_{\tau}^{y_0,\eta}(x_0-h+\lambda_2\xi h))\big)(\lambda_1\xi DX_{\tau}^{y_0,\eta}(x_0-h+\lambda_1\xi h)(h)) \\
&+ D\Psi(X_{\tau}^{y_0,\eta}(x_0-h+\lambda_2\xi h))\big(\lambda_1\xi DX_{\tau}^{y_0,\eta}(x_0-h+\lambda_1\xi h)(h)-\lambda_2\xi DX_{\tau}^{y_0,\eta}(x_0-h+\lambda_2\xi h)(h)\big)\big|^2\Big]\\
&\leq 2 \E\Big[\big|\big(D\Psi(X_{\tau}^{y_0,\eta}(x_0-h+\lambda_1\xi h))-D\Psi(X_{\tau}^{y_0,\eta}(x_0-h+\lambda_2\xi h))\big)(\lambda_1\xi DX_{\tau}^{y_0,\eta}(x_0-h+\lambda_1\xi h)(h))\big|^2 \Big]\\
&+ 2 \E\Big[\big|D\Psi(X_{\tau}^{y_0,\eta}(x_0-h+\lambda_2\xi h))\big(\lambda_1\xi DX_{\tau}^{y_0,\eta}(x_0-h+\lambda_1\xi h)(h)-\lambda_2\xi DX_{\tau}^{y_0,\eta}(x_0-h+\lambda_2\xi h)(h)\big)\big|^2\Big]\\
&\leq 2\E\Big[\big\|D\Psi(X_{\tau}^{y_0,\eta}(x_0-h+\lambda_1\xi h))-D\Psi(X_{\tau}^{y_0,\eta}(x_0-h+\lambda_2\xi h))\big\|_{\text{op}}^2\big\|\lambda_1\xi DX_{\tau}^{y_0,\eta}(x_0-h+\lambda_1\xi h)(h)\big\|_w^2 \Big]\\
&+ 2 \E\Big[\big\|D\Psi(X_{\tau}^{y_0,\eta}(x_0-h+\lambda_2\xi h))\big\|_{\text{op}}^2\big\|\lambda_1\xi DX_{\tau}^{y_0,\eta}(x_0-h+\lambda_1\xi h)(h)-\lambda_2\xi DX_{\tau}^{y_0,\eta}(x_0-h+\lambda_2\xi h)(h)\big\|_w^2\Big]\\
&\leq 2 L_{D\Psi}^2 \E\Big[\big\|X_{\tau}^{y_0,\eta}(x_0-h+\lambda_1\xi h)-X_{\tau}^{y_0,\eta}(x_0-h+\lambda_2\xi h)\big\|_w^2\big\|\lambda_1\xi DX_{\tau}^{y_0,\eta}(x_0-h+\lambda_1\xi h)(h)\big\|_w^2 \Big]\\
&+ 2 L_{\Psi}^2\E\Big[\big\|\lambda_1\xi DX_{\tau}^{y_0,\eta}(x_0-h+\lambda_1\xi h)(h)-\lambda_2\xi DX_{\tau}^{y_0,\eta}(x_0-h+\lambda_2\xi h)(h)\big\|_w^2\Big]\\
&\leq 2 L_{D\Psi}^2 \E\Big[\big\|(\lambda_1-\lambda_2)\xi\mS_\tau(h)\big\|_w^2\big\|\lambda_1\xi \mS_\tau(h)\big\|_w^2 \Big] + 2 L_{\Psi}^2\E\Big[\big\|(\lambda_1-\lambda_2)\xi\mS_\tau(h)\big\|_w^2\Big]\\
&\leq 2 L_{D\Psi}^2 |\lambda_1-\lambda_2|^2|\lambda_1|^2 \|\mS_\tau(h)\|_w^4\E\big[|\xi|^4\big] + 2 L_{\Psi}^2 |\lambda_1-\lambda_2|^2 \|\mS_\tau(h)\|_w^2 \E\Big[|\xi|^2\Big]\\
&\leq 2 L_{D\Psi}^2 |\lambda_1-\lambda_2|^2|\lambda_1|^2 \|h\|_w^4\E\big[|\xi|^4\big] + 2 L_{\Psi}^2 |\lambda_1-\lambda_2|^2 \|h\|_w^2 \E\big[|\xi|^2\big]\\
&= 2 \|h\|_w^2 \E\big[|\xi|^2\big] \left(L_{D\Psi}^2 |\lambda_1|^2 \|h\|_w^2\E\big[|\xi|^2\big] + L_{\Psi}^2  \right)|\lambda_1-\lambda_2|^2. 
\end{align*}
For all $y\in\R^+$, we then have
\begin{align*}
&\| \sD \left(u(\cdot,\lambda_1)-u(\cdot,\lambda_2)\right)(y) \|_{L^2(\Omega;H_w)}^2 \\
&= \big\| \sD \big(\left(\Psi(X_{\tau}^{y_0,\eta}(x_0-h+\lambda_1\xi h)) - \Psi(X_{\tau}^{y_0,\eta}(x_0-h+\lambda_2\xi h)) \right) h_x\big) (y)\big\|_{L^2(\Omega;H_w)}^2\\
&= \|h_x\|_w^2 \E\left[\big|\sD \left(\Psi(X_{\tau}^{y_0,\eta}(x_0-h+\lambda_1\xi h)) - \Psi(X_{\tau}^{y_0,\eta}(x_0-h+\lambda_2\xi h)) \right)(y) \big|^2\right] \\
&\leq 2 \|h_x\|_w^2 \|h\|_w^2 \E\big[|\xi|^2\big] \left(L_{D\Psi}^2 |\lambda_1|^2 \|h\|_w^2 \E\big[|\xi|^2\big] + L_{\Psi}^2 \right)|\lambda_1-\lambda_2|^2.
\end{align*}
Hence \eqref{kolmogorov-condition} is satisfied and the result follows by Lemma \ref{technical-lemma}.

\noindent\textbf{Proof of (ii):}
We define
\begin{equation*}
u(z,\lambda) \coloneqq \Psi(X_{\tau}^{x_0,\eta}(y_0-h+\lambda\xi h)) h_x(z)
\end{equation*}
for $z\in\R^+$ and $\lambda\in\R$.
Following the argument in the proof of (i), we need to show that
\begin{equation}\label{u-lip}
 \| u(\cdot,\lambda_1)-u(\cdot,\lambda_2) \|_{L^2(\Omega;H_w)}^2 < C_1|\lambda_1 - \lambda_2|^2
\end{equation}
and 
\begin{equation}\label{Du-lip}
 \| \sD \left(u(\cdot,\lambda_1)-u(\cdot,\lambda_2)\right) \|_{L^2(\Omega;H_w\x H_w)}^2 < C_2|\lambda_1 - \lambda_2|^2
\end{equation}
for some constants $C_1$ and $C_2$. 
Since $\Psi$ is Lipscitz with Lipschitz constant $L_{\Psi}$, we have that
\begin{align*}
\| u(\cdot,\lambda_1)-u(\cdot,\lambda_2) \|_{L^2(\Omega;H_w)}^2 &= \E\left[\|\left(\Psi(X_{\tau}^{x_0,\eta}(y_0-h+\lambda_1\xi h)) - \Psi(X_{\tau}^{x_0,\eta}(y_0-h+\lambda_2\xi h)) \right) h_x \|_w^2\right]\\
&= \|h_x\|_w^2 \E\left[\big|\Psi(X_{\tau}^{x_0,\eta}(y_0-h+\lambda_1\xi h)) - \Psi(X_{\tau}^{x_0,\eta}(y_0-h+\lambda_2\xi h))\big|^2\right] \\
&\leq \|h_x\|_w^2 L_{\Psi}^2 \E\left[\|X_{\tau}^{x_0,\eta}(y_0-h+\lambda_1\xi h) - X_{\tau}^{x_0,\eta}(y_0-h+\lambda_2\xi h)\|_w^2\right].
\end{align*}
By Hölder's inequality, the Burkholder-Davis-Gundy inequality, the It\^{o} isometry, the Hille-Yosida theorem and the calculations in Step 1, it holds that
\begin{align*}
&\E\left[\|X_{\tau}^{x_0,\eta}(y_0-h+\lambda_1\xi h) - X_{\tau}^{x_0,\eta}(y_0-h+\lambda_2\xi h)\|_w^2\right]\\
&= \E\left[\Big\|\int_0^\tau \mS_{\tau-s}\left(Z_s\x \mathcal{U}_s ((\lambda_1-\lambda_2)\xi h)\right)\,dB_s\Big\|_w^2\right]\\
&= |\lambda_1-\lambda_2|^2 \E\left[|\xi|^2\Big\|\int_0^\tau \mS_{\tau-s}\left(Z_s\x \mathcal{U}_s h\right)\,dB_s\Big\|_w^2\right]\\
&\leq |\lambda_1-\lambda_2|^2 \E\left[|\xi|^4\right]^{1/2}\E\left[\Big\|\int_0^\tau \mS_{\tau-s}\left(Z_s\x \mathcal{U}_s h\right)\,dB_s\Big\|_w^4\right]^{1/2}\\
&\leq C_{BDG} |\lambda_1-\lambda_2|^2 \E\left[|\xi|^4\right]^{1/2}\E\left[\left(\int_0^\tau \| \mS_{\tau-s}\left(Z_s\x \mathcal{U}_s h\right)Q_B^{1/2}\|_{\mH_w}^2\,ds\right)^2\right]^{1/2}\\
&\leq C_{BDG} |\lambda_1-\lambda_2|^2 \E\left[|\xi|^4\right]^{1/2} \E\left[\left(\int_0^\tau C^2\e^{2c(\tau-s)}\|(Z_s\x \mathcal{U}_s h) Q_B^{1/2}\|_{\mH_w}^2\,ds\right)^2 \right]^{1/2}\\
&\leq C_{BDG} |\lambda_1-\lambda_2|^2 \E\left[|\xi|^4\right]^{1/2} \E\left[\left(\int_0^\tau C^2\e^{2c(\tau-s)}K^2\e^{2ks} \|h\|_w^2 \text{Tr}(Q_B)\,ds\right)^2 \right]^{1/2}\\
&= C_{BDG} |\lambda_1-\lambda_2|^2 \E\left[|\xi|^4\right]^{1/2} C^2 K^2 \|h\|_w^2 \text{Tr}(Q_B)\int_0^\tau \e^{2c(\tau-s)} e^{2ks}\,ds.
\end{align*}
This shows that \eqref{u-lip} holds.
To show that \eqref{Du-lip} holds, we need to show that
\begin{equation}\label{eq1}
\E\left[\|X_{\tau}^{x_0,\eta}(y_0-h+\lambda_1\xi h) - X_{\tau}^{x_0,\eta}(y_0-h+\lambda_2\xi h)\|_w^2  \|\lambda_1 \xi DX_{\tau}^{x_0,\eta}(y_0-h+\lambda\xi h)(h)\|_w^2\right] \leq C|\lambda_1-\lambda_2|^2
\end{equation}
and
\begin{equation}\label{eq2}
\E\left[\|\lambda_1 \xi DX_{\tau}^{x_0,\eta}(y_0-h+\lambda\xi h)(h) - \lambda_2 \xi DX_{\tau}^{x_0,\eta}(y_0-h+\lambda\xi h)(h)\|_w^2\right] \leq C|\lambda_1-\lambda_2|^2.
\end{equation}
We see that
\begin{align*}
&\E\left[\|\lambda_1 \xi DX_{\tau}^{x_0,\eta}(y_0-h+\lambda\xi h)(h) - \lambda_2 \xi DX_{\tau}^{x_0,\eta}(y_0-h+\lambda\xi h)(h)\|_w^2\right] \\
&= \E\left[\Big\|(\lambda_1-\lambda_2)\xi \int_0^\tau \mS_{\tau-s}\left(Z_s\x \mathcal{U}_s h\right)\,dB_s \Big\|_w^2\right] \\
&= |\lambda_1-\lambda_2|^2\E\left[|\xi|^2 \Big\| \int_0^\tau \mS_{\tau-s}\left(Z_s\x \mathcal{U}_s h\right)\,dB_s \Big\|_w^2\right],
\end{align*}
hence \eqref{eq2} follows from the calculations above.
By Hölder's inequality, the Burkholder-Davis-Gundy inequality, the It\^{o} isometry and the Hille-Yosida theorem, it holds that
\begin{align*}
&\E\left[\|X_{\tau}^{x_0,\eta}(y_0-h+\lambda_1\xi h) - X_{\tau}^{x_0,\eta}(y_0-h+\lambda_2\xi h)\|_w^2  \|\lambda_1 \xi DX_{\tau}^{x_0,\eta}(y_0-h+\lambda\xi h)(h)\|_w^2\right]\\
&=\E\left[\Big\|\int_0^\tau \mS_{\tau-s}\left(Z_s\x \mathcal{U}_s ((\lambda_1-\lambda_2)\xi h)\right)\,dB_s\Big\|_w^2  \Big\|\lambda_1 \xi \int_0^\tau \mS_{\tau-s}\left(Z_s\x \mathcal{U}_s h\right)\,dB_s\Big\|_w^2\right]\\
&= \lambda_1^2 |\lambda_1-\lambda_2|^2 \E\left[|\xi|^4\Big\|\int_0^\tau \mS_{\tau-s}\left(Z_s\x \mathcal{U}_s h\right)\,dB_s\Big\|_w^4\right]\\
&\leq \lambda_1^2 |\lambda_1-\lambda_2|^2 \E\left[|\xi|^8\right]^{1/2}\E\left[\Big\|\int_0^\tau \mS_{\tau-s}\left(Z_s\x \mathcal{U}_s h\right)\,dB_s\Big\|_w^8\right]^{1/2}\\
&\leq C_{BDG} \lambda_1^2 |\lambda_1-\lambda_2|^2 \E\left[|\xi|^8\right]^{1/2}\E\left[\left(\int_0^\tau \|\mS_{\tau-s}\left(Z_s\x \mathcal{U}_s h\right)Q_B^{1/2}\|_w^2\,ds\right)^4 \right]^{1/2}.
\end{align*} 
Hence \eqref{eq1} holds from the calculations above.
Equation \eqref{Du-lip} then follows from the same argument as in the proof of (i).
The result then follows by Lemma \ref{technical-lemma}.

\noindent\textbf{Proof of (iii):}
We define
\begin{equation*}
u(z,\lambda) \coloneqq \Psi(X_{\tau}^{x_0,y_0}(\eta-\zeta+\lambda\xi \zeta)) h_x(z)
\end{equation*}
for $z\in\R^+$ and $\lambda\in\R$.
Following the argument in the proof of (i), we need to show that
\begin{equation}\label{u-lip2}
 \| u(\cdot,\lambda_1)-u(\cdot,\lambda_2) \|_{L^2(\Omega;H_w)}^2 < C_1|\lambda_1 - \lambda_2|^2
\end{equation}
and 
\begin{equation}\label{Du-lip2}
 \| \sD \left(u(\cdot,\lambda_1)-u(\cdot,\lambda_2)\right) \|_{L^2(\Omega;H_w\x H_w)}^2 < C_2|\lambda_1 - \lambda_2|^2
\end{equation}
for some constants $C_1$ and $C_2$. Since $\Psi$ is Lipscitz with Lipschitz constant $L_{\Psi}$, we have that
\begin{align*}
\| u(\cdot,\lambda_1)-u(\cdot,\lambda_2) \|_{L^2(\Omega;H_w)}^2 &= \E\left[\|\left(\Psi(X_{\tau}^{x_0,y_0}(\eta-\zeta+\lambda_1\xi \zeta)) - \Psi(X_{\tau}^{x_0,y_0}(\eta-\zeta+\lambda_2\xi \zeta)) \right) h_x \|_w^2\right]\\
&= \|h_x\|_w^2 \E\left[\big|\Psi(X_{\tau}^{x_0,y_0}(\eta-\zeta+\lambda_1\xi \zeta)) - \Psi(X_{\tau}^{x_0,y_0}(\eta-h+\lambda_2\xi \zeta))\big|^2\right] \\
&\leq \|h_x\|_w^2 L_{\Psi}^2 \E\left[\|X_{\tau}^{x_0,y_0}(\eta-\zeta+\lambda_1\xi \zeta) - X_{\tau}^{x_0,y_0}(\eta-\zeta+\lambda_2\xi \zeta)\|_w^2\right].
\end{align*}
Since $\xi$ is independent of $B$ and $W$, we have by the It\^{o} isometry that
\begin{align*}
&\E\left[\|X_{\tau}^{x_0,y_0}(\eta-\zeta+\lambda_1\xi \zeta) - X_{\tau}^{x_0,y_0}(\eta-\zeta+\lambda_2\xi \zeta)\|_w^2\right]\\
&= \E\left[\Big\|\int_0^\tau \mS_{\tau-s}\left(Z_s\x \int_0^s \mathcal{U}_{s-u}(\lambda_1-\lambda_2)\xi\zeta\, dW_u\right)\,dB_s\Big\|_w^2\right]\\
&= |\lambda_1-\lambda_2|^2 \E\left[|\xi|^2\Big\|\int_0^\tau \mS_{\tau-s}\left(Z_s\x \int_0^s \mathcal{U}_{s-u}\zeta\, dW_u\right)\,dB_s\Big\|_w^2\right]\\
&= |\lambda_1-\lambda_2|^2 \E\left[|\xi|^2\right]\E\left[\Big\|\int_0^\tau \mS_{\tau-s}\left(Z_s\x \int_0^s \mathcal{U}_{s-u}\zeta\, dW_u\right)\,dB_s\Big\|_w^2\right]\\
&= |\lambda_1-\lambda_2|^2 \E\left[|\xi|^2\right]\E\left[\int_0^\tau \left\| \mS_{\tau-s}\left(Z_s\x \int_0^s \mathcal{U}_{s-u}\zeta\, dW_u\right)Q_B^{1/2}\right\|_{\mH_w}^2\,ds\right].
\end{align*}
From the calculations in Step 1 it follows that \eqref{u-lip2} holds.
To show that \eqref{Du-lip2} holds, we need to show that
\begin{equation}\label{eq3}
\E\left[\|X_{\tau}^{x_0,y_0}(\eta-\zeta+\lambda_1\xi \zeta) - X_{\tau}^{x_0,y_0}(\eta-\zeta+\lambda_2\xi \zeta)\|_w^2  \|\lambda_1 \xi DX_{\tau}^{x_0,y_0}(\eta-\zeta+\lambda\xi \zeta)(\zeta)\|_w^2\right] \leq C|\lambda_1-\lambda_2|^2
\end{equation}
and
\begin{equation}\label{eq4}
\E\left[\|\lambda_1 \xi DX_{\tau}^{x_0,y_0}(\eta-\zeta+\lambda\xi \zeta)(\zeta) - \lambda_2 \xi DX_{\tau}^{x_0,y_0}(\eta-\zeta+\lambda\xi \zeta)(\zeta)\|_w^2\right] \leq C|\lambda_1-\lambda_2|^2.
\end{equation}
We see that
\begin{align*}
&\E\left[\|\lambda_1 \xi DX_{\tau}^{x_0,y_0}(\eta-\zeta+\lambda\xi \zeta)(\zeta) - \lambda_2 \xi DX_{\tau}^{x_0,y_0}(\eta-\zeta+\lambda\xi \zeta)(h\zeta)\|_w^2\right] \\
&= \E\left[\Big\|(\lambda_1-\lambda_2)\xi \int_0^\tau \mS_{\tau-s}\left(Z_s\x \int_0^s \mathcal{U}_{s-u}\zeta\, dW_u\right)\,dB_s \Big\|_w^2\right] \\
&= |\lambda_1-\lambda_2|^2\E\left[|\xi|^2 \Big\| \int_0^\tau \mS_{\tau-s}\left(Z_s\x \int_0^s \mathcal{U}_{s-u}\zeta\, dW_u\right)\,dB_s \Big\|_w^2\right],
\end{align*}
hence \eqref{eq4} follows from the calculations above.
Since $\xi$ is independent of $B$ and $W$, it follows from the Burkholder-Davis-Gundy inequality, Jensen's inequality and Fubini's theorem that
\begin{align*}
&\E\left[\|X_{\tau}^{x_0,y_0}(\eta-\zeta+\lambda_1\xi \zeta) - X_{\tau}^{x_0,y_0}(\eta-\zeta+\lambda_2\xi \zeta)\|_w^2  \|\lambda_1 \xi DX_{\tau}^{x_0,y_0}(\eta-\zeta+\lambda\xi \zeta)(\zeta)\|_w^2\right]\\
&=\E\left[\Big\|\int_0^\tau \mS_{\tau-s}\left(Z_s\x \int_0^s \mathcal{U}_{s-u}(\lambda_1-\lambda_2)\xi\zeta\, dW_u \right)\,dB_s\Big\|_w^2  \Big\|\lambda_1 \xi \int_0^\tau \mS_{\tau-s}\left(Z_s\x \int_0^s \mathcal{U}_{s-u}\zeta\, dW_u\right)\,dB_s\Big\|_w^2\right]\\
&= \lambda_1^2 |\lambda_1-\lambda_2|^2 \E\left[|\xi|^4\Big\|\int_0^\tau \mS_{\tau-s}\left(Z_s\x \mathcal{U}_s h\right)\,dB_s\Big\|_w^4\right]\\
&= \lambda_1^2 |\lambda_1-\lambda_2|^2 \E\left[|\xi|^4\right]\E\left[\Big\|\int_0^\tau \mS_{\tau-s}\left(Z_s\x \int_0^s \mathcal{U}_{s-u}\zeta\, dW_u\right)\,dB_s\Big\|_w^4\right]\\
&\leq C \lambda_1^2 |\lambda_1-\lambda_2|^2 \E\left[|\xi|^4\right]\E\left[\left(\int_0^\tau \left\|\mS_{\tau-s}\left(Z_s\x \int_0^s \mathcal{U}_{s-u}\zeta\, dW_u\right)Q_B^{1/2}\right\|_w^2\,ds\right)^2 \right]\\
&\leq C \lambda_1^2 |\lambda_1-\lambda_2|^2 \E\left[|\xi|^4\right]\E\left[\int_0^\tau \left\|\mS_{\tau-s}\left(Z_s\x \int_0^s \mathcal{U}_{s-u}\zeta\, dW_u\right)Q_B^{1/2}\right\|_w^4\,ds \right]\\
&\leq C \lambda_1^2 |\lambda_1-\lambda_2|^2 \E\left[|\xi|^4\right]\E\left[\int_0^\tau \|\mS_{\tau-s}\|_{\text{op}}^4\left\|\left(Z_s\x \int_0^s \mathcal{U}_{s-u}\zeta\, dW_u\right)Q_B^{1/2}\right\|_w^4\,ds \right]\\
&= C \lambda_1^2 |\lambda_1-\lambda_2|^2 \E\left[|\xi|^4\right]\int_0^\tau \|\mS_{\tau-s}\|_{\text{op}}^4 \E\left[\left\|\left(Z_s\x \int_0^s \mathcal{U}_{s-u}\zeta\, dW_u\right)Q_B^{1/2}\right\|_w^4 \right]\,ds.
\end{align*} 
From the calculations in Step 1, we have that
\begin{align*}
\left\|\left(Z_s\x \int_0^s \mathcal{U}_{s-u}\zeta\, dW_u\right)Q_B^{1/2}\right\|_w^4 \leq \text{Tr}(Q_B)^2 \left\|\int_0^s \mathcal{U}_{s-u}\zeta\, dW_u\right\|_w^4.
\end{align*}
The Burkholder-Davis-Gundy inequality, Jensen's inequality and the Hille-Yosida theorem give that
\begin{align*}
\E\left[\left\|\int_0^s \mathcal{U}_{s-u}\zeta\, dW_u\right\|_w^4\right] &\leq C \E\left[\left(\int_0^s \|\mathcal{U}_{s-u}\zeta Q_W^{1/2}\|_{\mH_w}^2\,ds\right)^2\right]\\
&\leq C \E\left[\int_0^s \|\mathcal{U}_{s-u}\zeta Q_W^{1/2}\|_{\mH_w}^4\,ds\right]\\
&\leq C \E\left[\int_0^s \|\mathcal{U}_{s-u}\|_{\text{op}}^4 \|\zeta\|_{\text{op}}^4 \text{Tr}(Q_W)^2\,ds\right]\\
&\leq C \text{Tr}(Q_W)^2 \|\zeta\|_{\text{op}}^4 \E\left[\int_0^s K^4 e^{4k(s-u)}  \,ds\right].
\end{align*}
Hence \eqref{eq3} holds.
Equation \eqref{Du-lip2} then follows from the same argument as in the proof of (i).
The result then follows by Lemma \ref{technical-lemma}.
\end{proof}

In \cite[Section 3.3]{BDHP} the authors generalize the expression for the delta in their model to payoff functions which are not Fr\'{e}chet differentiable. In this case they replace the payoff function by a Moreau-Yosida approximation and take the limit to obtain the delta.
It is expected that the same argument can be used to generalize our results to non-smooth payoff functions in a similar way. 

\section*{Acknowledgement}
The research leading to this work has received support from The Research Council of Norway via the project STORM: Stochastics for Time-Space Risk Models (nr. 274410). 
We are grateful to Espen Sande for many useful discussions.

\end{document}